\documentclass[11pt,a4paper]{amsart}
\pdfoutput=1
\usepackage[utf8]{inputenc}
\usepackage{enumerate}
\usepackage{paralist}
\usepackage{mathrsfs}
\usepackage{color}
\usepackage[pdftex]{graphicx}
\usepackage{amssymb}
\usepackage{amsmath,amsthm}
\usepackage{hyperref}
\usepackage{soul}
\hypersetup{
  colorlinks=false,
  plainpages=false,
}

\DeclareGraphicsExtensions{pdf, jpg, jpeg, png}

\graphicspath{{./figures/}}

\newtheorem{theorem}{Theorem}
\newtheorem{lemma}[theorem]{Lemma}
\newtheorem{proposition}[theorem]{Proposition}

\newtheorem{definition}[theorem]{Definition}
\newtheorem{remark}[theorem]{Remark}
\theoremstyle{remark}

\renewcommand{\paragraph}[1]{\medskip\noindent\textbf{#1}}

\newcommand{\id}{\mathrm{id}}
\newcommand{\R}{\mathbb{R}}

\newcommand{\Z}{\mathbb{Z}}

\newcommand{\RP}{\mathbb{R}\textup{P}}

\renewcommand{\P}{\operatorname{P}}

\newcommand{\plspace}{\mathbb{R}\textup{P}^{3,3}}
\newcommand{\plquad}{\mathcal{Q}^{3,3}}
\newcommand{\plcone}{\mathbb{L}^{3,3}}
\newcommand{\plisolines}{\mathscr{L}^{3,3}}
\newcommand{\hypplanes}{\mathcal{P}_{\textrm{hyp}}}
\newcommand{\rp}[1]{\mathsf{#1}}
\newcommand{\rpc}[1]{\what{\mathsf{#1}}}
\newcommand{\pc}[1]{\what{#1}}

\newcommand{\pol}[1]{\operatorname{pol}[#1]}
\newcommand{\inc}[1]{\operatorname{inc}[#1]}
\newcommand{\vspan}[1]{\operatorname{span}(#1)}

\renewcommand{\phi}{\varphi}
\renewcommand{\epsilon}{\varepsilon}

\newcommand{\what}{\widehat}
\renewcommand{\implies}{\Rightarrow}

\usepackage{caption2}
\setcaptionwidth{.9\textwidth}

\title[Extension of discrete A-nets using hyperboloid patches]{
  Discretization of\\
  asymptotic line parametrizations\\
  using hyperboloid surface patches
}

\author{Emanuel Huhnen-Venedey \and Thilo R\"{o}rig} %
\address{Emanuel Huhnen-Venedey\vspace{-5pt}}
\address{Inst.~Mathematics MA 8-3, TU Berlin, 10623 Berlin, Germany}%
\email{huhnen@math.tu-berlin.de}

\address{Thilo R\"{o}rig\vspace{-5pt}} 
\address{Inst.~Mathematics MA 8-3, TU Berlin, 10623 Berlin, Germany}
\email{roerig@math.tu-berlin.de}

\thanks{The first author was partially supported by the German Academic Exchange Service (DAAD) and the University of New South Wales in Sydney.}


\begin{document}

\begin{abstract}
  Two-dimensional affine A-nets in $3$-space are quadrilateral meshes that
  discretize surfaces pa\-rametrized along asymptotic lines. The characterizing
  property of A-nets is planarity of vertex stars, so for generic A-nets the
  elementary quadrilaterals are skew. We classify the simply connected
  affine A-nets that can be extended to continuously differentiable surfaces by
  gluing hyperboloid surface patches into the skew quadrilaterals.  The
  resulting surfaces are called ``hyperbolic nets'' and are a novel
  piecewise smooth discretization of surfaces parametrized along asymptotic
  lines. It turns out that a simply connected affine A-net has to satisfy one
  combinatorial and one geometric condition to be extendable --
  all vertices have to be of even degree and all quadrilateral strips have to be ``equi-twisted''.
  Furthermore, if an A-net can be extended to a hyperbolic net, then there
  exists a 1-parameter family of such $C^1$-surfaces.
  It is briefly explained how the generation of hyperbolic nets can be implemented
  on a computer. The article uses the projective model
  of Pl\"ucker geometry to describe A-nets and hyperboloids.
\end{abstract}

\maketitle

\section{Introduction}
\label{sec:intro}

The present paper deals with the discretization of surfaces in 3-space that are
parametrized along asymptotic lines. Usually the discretization of parametrized
surfaces within discrete differential geometry (DDG) leads to quadrilateral
nets, often also called quadrilateral meshes.  
Such nets not only discretize smooth (point) sets, but also 
reflect the combinatorial structure of the parametrizations to be discretized. This generalizes to the case of nets of
arbitrary dimension with elementary 2-cells being quadrilaterals. These are omnipresent
in DDG as discretizations of various parametrized geometries.

While general quadrilateral nets discretize arbitrary parametrizations, the
discretization of distinguished parametrizations yields quadrilateral nets with
special geometric properties. The most fundamental example is the
discretization of conjugate parametrizations by quadrilateral nets with planar
faces. Discretizing more specific conjugate parametrizations yields
planar quadrilateral nets with additional properties. For instance, different
discretizations of curvature line parametrizations, which are
orthogonal conjugate parametrizations, have led to the notions of
circular and conical nets \cite{Bobenko:1999:DiscreteConformalMaps,
CieslinskiDoliwaSantini:1997:CircularNets,KonopelchenkoSchief:1998:3DIntegrableLattices,LiuPottmannWallnerYangWang:2006:ConicalNets}.
Circular and conical nets are unified within the
framework of principle contact element nets
\cite{BobenkoSuris:2007:OrganizingPrinciples,PottmannWallner:2007:FocalGeometry}.
The latter are not point maps anymore but take contact elements as
values at vertices of the domain.

Curvature line parametrizations of smooth surfaces exist (at least locally)
around non-umbilic points and are essentially unique (up to reparametrization
of the parameter lines).  In perfect analogy, one has unique asymptotic line
parametrizations around hyperbolic points of surfaces, that is, around points
of negative Gaussian curvature. Indeed, the description of curvature lines in
Lie geometry is equivalent to the description of asymptotic lines in Pl\"ucker
geometry, see, e.g., \cite{Klein:1926:GeometrieVorlesungen}.  In both
geometries a surface is described by its contact elements. In Lie geometry a
contact element consists of all spheres touching in a point, in Pl\"ucker
geometry it is composed of all lines in a plane through a point in that plane.
In both cases it is convenient to imagine a contact element as a 2-plane
containing a distinguished point. Curvature lines are characterized in Lie
geometry by the fact, that ``infinitesimally close contact elements along a
curvature line'' through a point~$p$ share a sphere, which is the principal
curvature sphere of this curvature line in $p$. Completely analogous, in Pl\"ucker
geometry two infinitesimally close contact elements along an asymptotic line
through a point $p$ share a line, which is the tangent to this asymptotic line
in $p$. However, as
asymptotic line parametrizations are not conjugate parametrizations, they are
not modelled by quadrilateral nets with planar faces. Instead, asymptotic line
parametrizations are properly discretized by quadrilateral nets with planar
vertex stars, that is, nets for which every vertex is coplanar with its next
neighbors.  We use the terminology of \cite{BobenkoSuris:2008:DDGBook}, calling
nets with planar quadrilaterals \emph{Q-nets} and nets with planar vertex stars
\emph{A-nets}. Q-nets and A-nets as discretizations of conjugate and asymptotic
line parametrizations were already introduced
in~\cite{Sauer:1937:ProjLinienGeometrie}. To be precise, the modern description
of A-nets in $\RP^3$ does not involve distinguished edges, i.e., line segments
connecting adjacent vertices. However, if we speak about an A-net as a quadrilateral
mesh, we have distinguished edges in mind. Therefore we use the term \emph{affine A-net}
to reflect the idea, that choosing an ideal plane at infity in $\RP^3$ determines
unique finite edges to be used to connect adjacent edges of an A-net.

Various aspects of smooth asymptotic line parametrizations have been
discretized using A-nets, often with a focus on
preserving the relations of the classical theory.  For example, the
discretization of surfaces of constant negative Gaussian curvature as special
A-nets, nowadays called K-surfaces, can be found in
\cite{Sauer:1950:Pseudosphaeren,Wunderlich:1951:K-Flaechen}. Much later, in
context of the connections between geometry and integrability, Bobenko and Pinkall
\cite{BobenkoPinkall:1996:DiscreteKandHirota} established the relation between
K-surfaces and the discrete sine-Gordon equation set down by Hirota
\cite{Hirota:1977:DiscreteSineGordon}.  For a
special instance of this relation, see, for example,
\cite{Hoffmann:1999:DiscreteAmsler} on discrete Amsler-surfaces, which are
special K-surfaces.
Discrete indefinite affine spheres
\cite{BobenkoSchief:1999:AffineSpheresIndefinite} are an example of the
discretization of a certain class of smooth A-nets within affine differential
geometry.  The discrete Lelieuvre representation of A-nets and the related
discrete Moutard equations are, for instance, treated in
\cite{KonopelchenkoPinkall:2000:ProjectiveLelieuvre,NimmoSchief:1997:MoutardSuperposition}.

An essential aspect of specific surface parametrizations is the corresponding
class of transformations together with the associated permutability theorems.
In the case of smooth A-nets, the
associated transformations are called Weingarten transformations. Two smooth
A-surfaces are said to be Weingarten transforms of each other if the line
connecting corresponding points is the intersection of the tangent planes to
the two surfaces at these points.  This relation carries over naturally to the
discrete setting
\cite{Doliwa:2001:ANetsPluecker,DoliwaNieszporskiSantini:2001:IntegrableReductionOfAnets}.

In the present work we show how to extend discrete affine A-nets to piecewise smooth
$C^1$-surfaces. The resulting surfaces, called \emph{hyperbolic nets}, are a
novel discretization of surfaces parametrized along asymptotic lines.  For our
approach, the description of A-nets in the projective model of Pl\"ucker
geometry is crucial \cite{Doliwa:2001:ANetsPluecker}. In this setting, A-nets
appear as discrete congruences of isotropic lines, that is, lines that are
completely contained in the Pl\"ucker quadric. In other words, A-nets in
$\RP^3$ are described in terms of contact elements, since isotropic lines in
the Pl\"ucker quadric represent contact elements in $\RP^3$. Indeed the
description of A-nets by their contact elements is a literal discretization of
the aforementioned characterizing property of asymptotic lines in Pl\"ucker
geometry.\footnote{ In the same spirit, the mentioned principle contact element
nets are a literal discretization of the characterizing property of curvature
lines in Lie geometry}  The extension of
A-nets to hyperbolic nets uses surface patches taken from hyperboloids, where a
hyperboloid in our sense is a doubly ruled quadric, i.e., a one-sheeted
hyperboloid or a hyperbolic paraboloid from the affine viewpoint.  The hyperboloid patches,
bounded by (straight) asymptotic lines of the supporting surfaces, are inserted
into the skew quadrilaterals of the A-nets such that the boundaries of the
patches align with the edges of the quadrilaterals.  Furthermore, two patches
sharing the common edge of edge-adjacent quadrilaterals have coinciding tangent
planes along this edge. This yields the $C^1$-property, if
the case that two adjacent patches form a cusp is excluded. The discretization of
asymptotic line parametrized surfaces as hyperbolic nets is very similar to the
discretization of curvature line parametrized surfaces by cyclidic nets,
introduced in \cite{BobenkoHuhnen-Venedey:2011:cyclidicNets}.
Weingarten transformations of hyperbolic nets have also been
investigated and will be the topic of a forthcoming paper.

There is only little work about surfaces composed of hyperboloid patches. 
Discrete affine minimal surfaces, which are special A-nets, are extended by
hyperboloid patches in~\cite{Craizer:2010:AffineMinimalSurfaces}. Starting with
the discrete Lelieuvre representation
of A-nets, Craizer et al.\ construct a bilinear extension of the considered
discrete surfaces.  This yields continuous surfaces composed of hyperboloid
patches that are adapted to the underlying discrete A-net. It turns out that, in the
very special situation of discrete affine minimal surfaces, this extension
indeed yields a piecewise smooth $C^1$-surface, i.e., a hyperbolic net in our
sense.  However, for general A-nets the extension yields only continuous, but
not differentiable surfaces.

It is natural to look for applications of hyperbolic nets in the context of
Computer Aided Geometric Design and architectural geometry. The latter is a
emerging field of applied mathematics, which provides
the architecture community with sophisticated geometric knowledge to tackle
diverse problems.
Focussing also on important aspects such as efficient manufacturing, providing intuitive
control of available degrees of freedom, and similar issues, many results in
DDG have already been applied in an architectural context.  Conversely,
problems in architecture have inspired the investigation of
particular geometric configurations.
We are sure that the present article will also find application in this field 
as there is growing interest in the extension of discrete support
structures to surfaces built from curved panels
\cite{BoEtAl:2011:CAS,PottmannEtAl:2008:FreeformSurfaceSingleCurvedPanels,LiuPottmannWallnerYangWang:2006:ConicalNets}.

\paragraph{Structure of the article and main results.}
Section~\ref{sec:proj_and_pluecker_geo} introduces our notations and the projective model of
Pl\"ucker line geometry. Discrete A-nets in $\RP^3$ are discussed in
Section~\ref{sec:preHypNets} and their description as discrete line congruences
contained in the Pl\"ucker quadric is given. Aiming at the extension of
A-nets to hyperbolic nets, we introduce pre-hyperbolic nets as an intermediate
step.  While hyperbolic nets contain a hyperboloid patch for each elementary
quadrilateral, pre-hyperbolic nets contain the whole supporting hyperboloid.
The main result of Section~\ref{sec:preHypNets} is that any simply connected
A-net with interior vertices of even degree can be extended to a pre-hyperbolic
net. In this case, there exists a 1-parameter family of pre-hyperbolic nets.
An extension is determined by the choice of an initial
hyperboloid associated with an arbitrary elementary quadrilateral. The initial
hyperboloid is propagated to all other quadrilaterals. The essential
step is to show the consistency of this propagation.  Since we are dealing with
simply connected topology only, there are no global closure conditions for the
evolution. Hence, we have to show consistency of the propagation around a
single inner vertex only. It turns out that a necessary and sufficient condition for
the consistency is even vertex degree.  In Section~\ref{sec:hyperbolic_nets}, we turn to
discrete affine A-nets, i.e., A-nets whose vertices are connected by finite
edges, and characterize those nets that allow for an extension by hyperboloid
patches to piecewise smooth $C^1$-surfaces.  The additional characterizing
property of affine A-nets that can be extended to hyperbolic nets, besides
inner vertices being of even vertex degree, turns out to be ``equi-twist'' of
quadrilateral strips. At the end of Section~\ref{sec:hyperbolic_nets}, we comment on the
computer implementation of our results. The appendix 
provides some calculations in coordinates.

\section{Projective and Pl\"ucker geometry}
\label{sec:proj_and_pluecker_geo}

There exists a beautiful description of discrete A-nets and hyperboloids in the
projective model of Pl\"ucker geometry.\footnote{The description of
hyperboloids in the projective model of Pl\"ucker geometry is completely analog
to the description of Dupin cyclides in the projective model of Lie geometry,
see Lie's famous line-sphere correspondence.}
Before discussing this description, we introduce the basic concepts and
notations for projective geometry and Pl\"ucker line geometry.

\subsection{Projective geometry and Plücker geometry background}
\label{subsec:background}

Classical Plücker geometry is the geometry of lines in the projective $3$-space
$\RP^3 = \P(\R^4)$.  In the projective model of Plücker geometry those lines
are represented as points on the 4-dimensional Plücker quadric $\plquad$, which
is embedded in the 5-dimensional projective space $\plspace = \P(\R^{3,3})$,
where $\R^{3,3}$ is $\R^6$ equipped with the inner product defining the Plücker
quadric.

For an introduction to projective geometry and quadrics in projective spaces,
see, e.g., \cite{Audin:2003:Geometry}.

\paragraph{Notation.} For the readers convenience we use different fonts to
distinguish the projective spaces $\RP^3$ and $\plspace$.  In $\RP^3$ we denote
points $\rp x \in \RP^3$, lines $\rp h \subset \RP^3$, and planes $\rp P
\subset \RP^3$, while objects in $\plspace$ are written in normal math font,
i.e., $h \in \plspace$ etc.  Homogeneous coordinates of both $\RP^3$ and
$\plspace$ are marked with a hat.  With respect to $\RP^3$ we write for example
$\rp x = [\rpc x] = \P(\R \rpc x)$ and for $\plspace$ the notation $h = [\pc h]
= \P(\R \pc h)$ is used.  More general, for a projective subspace $U$ the
associated linear subspace is denoted by~$\pc U$.

The inclusion minimal projective subspace containing projective subspaces $U_i =
\P(\pc U_i)$, $i=1, \dots, n$ is the \emph{projective span} of the $U_i$:
\begin{equation*}
\inc{U_1,\dots,U_n} := \P(\vspan{\pc U_1,\dots,\pc U_n}).
\end{equation*}

As mentioned before, the projective space $\plspace$ is equipped with the
Plücker quadric $\plquad$.  Polar subspaces with respect to $\plquad$ are
denoted by 
\begin{equation*} \pol{U_1,\dots,U_n} := \pol{\inc{U_1,\dots,U_n}} =
\P(\vspan{\pc U_1,\dots,\pc U_n}^\perp), \end{equation*} where $\perp$ denotes
the orthogonal complement with respect to the inner product on $\R^{3,3}$
defining the Plücker quadric.

\paragraph{Projective model of Plücker geometry.}
There exist a lot of books dealing with the projective model of Pl\"ucker
geometry. A classical reference for details on Plücker coordinates and the
Plücker inner product is~\cite{Klein:1926:GeometrieVorlesungen}. For a modern
treatment see, e.g., 
\cite{BobenkoSuris:2008:DDGBook,PottmannWallner:2001:ComputationalLineGeometry}.

The \emph{Plücker inner product} is a symmetric bilinear form
$\langle \cdot , \cdot \rangle : \R^6 \times \R^6 \to \R$ of signature $({+}{+}{+}{-}{-}{-})$.
We denote $\R^6$ equipped with this product by $\R^{3,3}$ and write the corresponding null-vectors in $\R^{3,3}$ as
\begin{equation*}
  \plcone := 
  \left\{ 
  \widehat v \in \R^{3,3} \mid 
  \langle \widehat v,\widehat v \rangle = 0
  \right\}.
\end{equation*}
Projectivization of $\plcone$ yields the \emph{Plücker quadric}
\begin{equation*}
  \plquad := \P(\plcone) \subset \textup{P}(\mathbb{R}^{3,3}) =: \plspace.
\end{equation*}
An element $h = [\pc h] \in \plspace$ is contained in $\plquad$ if and only if
$\pc h \in \R^{3,3}$ are homogeneous Pl\"ucker coordinates of a line
$\rp h \subset \RP^3$. Two lines $\rp h_1,\rp h_2$ in $\RP^3$
intersect if and only if their representatives $h_1, h_2 \in \plquad$ are polar
with respect to the Plücker quadric, i.e.,~$\langle \pc h_1,\pc h_2 \rangle =
0$. 

\begin{remark}
The projective model of Plücker geometry is often formulated in the language of
exterior algebra, see, for instance,
\cite{BobenkoSuris:2008:DDGBook,PottmannWallner:2001:ComputationalLineGeometry,Doliwa:2001:ANetsPluecker}.
One advantage of the exterior algebra formulation is the immediate presence of
Pl\"ucker line coordinates and easy calculation with them.
On the other hand, the formulation chosen for the present article
is more elementary and, most notably, points out the deep relation between
Pl\"ucker line geometry and Lie sphere geometry, cf.
\cite{BobenkoHuhnen-Venedey:2011:cyclidicNets}.\footnote{The projective model
of (3-dimensional) Lie geometry can be realized in $\P(\R^{4,2})$, i.e., using
an inner product of signature (4,2) instead one of signature (3,3) on $\R^6$.}
Benefiting from both descriptions, in Appendix~\ref{appendix}
we briefly translate between the two languages and use the exterior algebra
formulation to carry out some calculations. In the main text we are able to avoid
the introduction of explicit line coordinates in $\R^{3,3}$, since for our
considerations only the existence of an inner product with the
mentioned properties is relevant. In particular, we only need to describe polarity
with respect to the Pl\"ucker quadric and make general statements about the
signature of the Pl\"ucker product restricted to projective subspaces.
\end{remark}

Two lines $\rp h_1,\rp h_2 \subset \RP^3$ intersecting in a point $\rp x \in
\RP^3$ span a \emph{contact element}, which consists of all lines through $\rp
x$ lying in the plane spanned by~$\rp h_1$ and~$\rp h_2$.  In the projective
model of Plücker geometry this contact element is the projective line $L =
\inc{h_1,h_2}$, and since $\langle \pc h_1 , \pc h_2 \rangle = 0$ one has~$L
\subset \plquad$. Lines contained in a quadric are commonly called
\emph{isotropic lines} and we define
\begin{equation*}
  \plisolines = 
  \left\{ \text{Isotropic lines in } \plspace \right\} 
  \cong \left\{ \text{Contact elements of } \RP^3 \right\}.
\end{equation*}
It is convenient to think of a contact element in $\RP^3$ as a plane $\rp P$
with a distinguished point $\rp x$ (see Fig.~\ref{fig:contact_elements}) and
write $(\rp x,\rp P)$ for $L \in \plisolines$.

The \emph{signature of a projective subspace} $U = \P(\pc U) \subset \plspace$
is the signature of $\langle \cdot , \cdot \rangle$ restricted to $\pc U$. In
particular the signature of a contact element, i.e. an isotropic line, is $({0}{0})$.

\begin{figure}[t]
  \begin{center}
     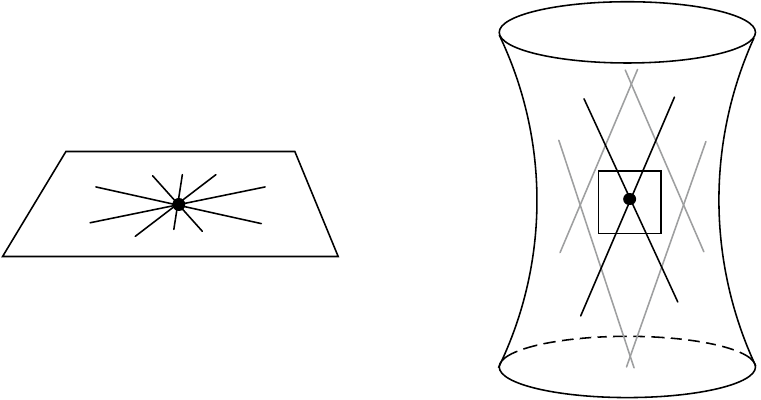 
  \end{center}
  \caption{\emph{Left}: A contact element consists of all lines in a plane passing through a
  distinguished point. \emph{Right}: For a hyperboloid $\mathcal{H}$, two asymptotic lines
	of different families intersect in a unique point and span the corresponding
    contact element of $\mathcal{H}$ (the supporting plane being tangent to
    $\mathcal{H}$).}
  \label{fig:contact_elements}
\end{figure}

\begin{remark}
  Two skew lines in $\RP^3$ do not belong to a common contact element and hence
  define a line of signature $(+-)$ in $\plspace$.  In particular the three
  pairs of opposite edges of a non-planar tetrahedron in $\RP^3$ yield three
  lines of signature $(+-)$, which toghether span the whole $\plspace$.
  \label{rem:tetrahedron}
\end{remark}

\subsection{Hyperboloids in Plücker geometry}
\label{subsec:hyperboloids_pluecker_description}

A large part of this article is concerned with projective geometry. So we use
the term \emph{hyperboloid} for an arbitrary doubly ruled quadric in $\RP^3$.
In particular, a hyperboloid of one sheet and a hyperbolic paraboloid, being
different in affine geometry, are indistinguishable in projective geometry.
Referring to an affine chart of~$\RP^3$ one may say, that a hyperbolic
paraboloid is a hyperboloid of one sheet that is tangent to the ideal plane at
infinity. Looking at the intersection of a hyperboloid with the ideal plane,
an affine hyperboloid of one sheet intersects it in a
non-degenerate conic, while a hyperbolic paraboloid intersects it in two
intersecting lines, i.e., a certain degenerate conic (see,
e.g.,~\cite{Audin:2003:Geometry}).

An important statement on quadrics in $\RP^3$ is, that any three
mutually skew lines determine a unique hyperboloid.  From the previous
considerations it follows, that in a fixed affine chart this hyperboloid is a
hyperbolic paraboloid, if and only if the three intersection points of the
considered lines with the ideal plane are collinear.

The following description of hyperboloids in the projective model of Plücker
geometry can be found, for example, in \cite{Klein:1926:GeometrieVorlesungen}.

\begin{theorem}
  Hyperboloids in $\RP^3$ are in bijection with polar decompositions of
  $\plspace$ into two projective planes of signatures $({+}{+}{-})$ and
  $({+}{-}{-})$.
  \label{thm:hyperboloids_pluecker} 
\end{theorem}

Usually we denote the two (disjoint) planes corresponding to a hyperboloid by
$P^{(1)}$ and $P^{(2)}$.  The conic sections $\plquad \cap P^{(1)}$ and
$\plquad \cap P^{(2)}$ describe of the two families of
rulings of a hyperboloid. Two lines $h^{(1)} \in P^{(1)}$ and $h^{(2)} \in
P^{(2)}$, one of each family, define a unique contact element of the
hyperboloid~(cf.~Figs.~\ref{fig:contact_elements} and \ref{fig:hyperbolic_planes}). 

\begin{definition}[Hyperbolic family of lines / regulus / ruling]
  Denote
  \[ 
    \hypplanes = 
    \left\{ 
    \text{Planes in } \plspace \text{ of signature }
    ({+}{+}{-}) \text{ or }({+}{-}{-})
    \right\}.
  \]
  A 1-parameter family of lines corresponding to a conic section $P \cap \plquad$
  with $P \in \hypplanes$ is called a \emph{hyperbolic family of lines} (cf.\
  Fig.~\ref{fig:hyperbolic_planes}).  We also use the term \emph{regulus} for a
  hyperbolic family of lines.
	A line of a regulus is called a \emph{ruling} of the corresponding hyperboloid.
  \label{def:hyperbolic_planes}
\end{definition}

\begin{figure}[t]
  \begin{center}
     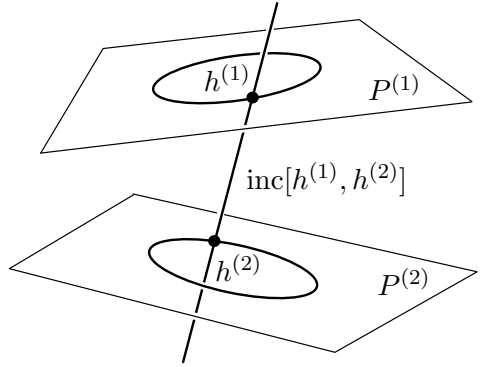 
		\caption{Polar planes $P^{(1)},P^{(2)} \in \hypplanes$ describe a
		hyperboloid $\mathcal{H} \subset \RP^3$: Each plane intersects the Plücker
		quadric $\plquad$ in a non-degenerate conic $\mathcal{H}^{(i)} = \plquad
		\cap P^{(i)}$, and each conic corresponds to one regulus of
		$\mathcal{H}$.  Each point~$\rp x$ of the hyperboloid is obtained as
		intersection point of lines $\rp h_1$ and $\rp h_2$ with~$h^{(i)} \in
		\mathcal{H}^{(i)}$.  The isotropic line $\inc{h^{(1)},h^{(2)}} \subset
		\plquad$ is the corresponding contact element of $\mathcal{H}$ at $\rp x$.}
    \label{fig:hyperbolic_planes}
  \end{center}
\end{figure}

\begin{remark}
  The complementary signatures $({+}{+}{-})$ and $({+}{-}{-})$ in
  Definition~\ref{def:hyperbolic_planes} reflect the two possible orientations
  of hyperbolic families of lines in $\RP^3$,
  cf.~\cite{Klein:1926:GeometrieVorlesungen}. In particular the two reguli of a
  hyperboloid in $\RP^3$ are of opposite orientation, according to the
  complementary signatures of the 2-planes introduced in
  Theorem~\ref{thm:hyperboloids_pluecker}. The different orientations of reguli
  will become important in the context of hyperbolic nets, so a more detailed
	discussion follows in Section~\ref{sec:hyperbolic_nets}.
  \label{rem:orientation_of_reguli}
\end{remark}

\section{Pre-hyperbolic nets}
\label{sec:preHypNets}

In this section, we start out with two equivalent definitions of discrete
A-nets. Then we show how to attach hyperboloids to the elementary
quadrilaterals of a discrete A-net, such that hyperboloids of edge-adjacent
quadrilaterals are tangent along a common asymptotic line that is determined by
the A-net. In Theorem~\ref{thm:nets_of_hyperboloids_from_anets} we show that
there exists a 1-parameter family of such adapted hyperboloids for any discrete
A-net with all inner vertices of even degree. Such A-nets with adapted
hyperboloids, called pre-hyperbolic nets, are the first step towards hyperbolic
nets to be introduced in Section~\ref{sec:hyperbolic_nets}.

\subsection{Discrete A-nets}
\label{subsec:anets}
Two dimensional discrete A-nets are a
discretization of smooth surfaces parametrized along asymptotic lines, going
back to Sauer \cite{Sauer:1937:ProjLinienGeometrie}.  We use two equivalent
definitions of A-nets as maps on a quad-graph, one in terms of the
vertices and another one in terms of contact elements (cf.
\cite{Doliwa:2001:ANetsPluecker,BobenkoSuris:2008:DDGBook}).

\begin{definition}[Quad-graph / vertex star]
A \emph{quad-graph} is a strongly regular polytopal cell decomposition of a
surface, such that all faces are quadrilaterals. We write $\mathcal{D}=\left( V,E,F
\right)$, where $V$ is the set of vertices (0-cells), $E$ is the set of edges
(1-cells), and $F$ is the set of faces (2-cells) of the quad-graph.

Let $v \in V$ be a vertex of $\mathcal{D}$. The \emph{vertex star} of $v$ is
the set of all vertices that are adjacent to $v$, including $v$ itself.
\label{def:quadgraph}
\end{definition}

\begin{definition}[Discrete A-net] Let $\mathcal{D}$ be a quad-graph with
  vertices~$V$.
  \begin{itemize}[$\bullet$]
  \item \emph{Vertex description of A-nets:}
    Let $\rp f : V \to \RP^3$. Then $\rp f$ is a discrete A-net if all
    vertex stars are planar, cf.\ Fig.~\ref{fig:planar_star}. 
  \item \emph{Contact element description of A-nets:}
    Let $L : V \to \plisolines \subset \plspace$. Then~$L$ defines the contact elements of a
    discrete A-net if adjacent isotropic lines intersect. In other words,~$L$ is a
    discrete congruence of (isotropic) lines in $\plspace$.
  \end{itemize}
  \label{def:anet}
\end{definition}
\begin{figure}[t]
  \begin{center}
    \includegraphics[scale=.16]{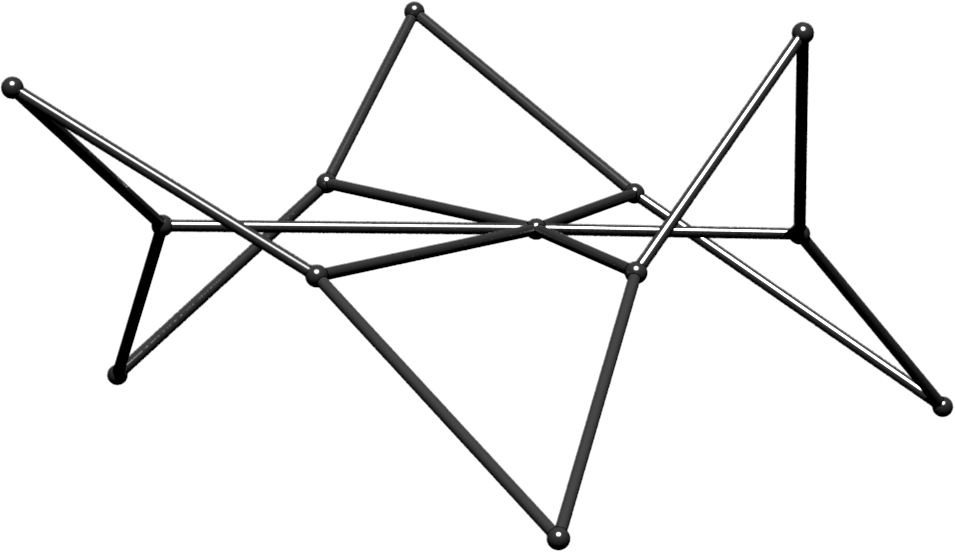}
  \end{center}
  \caption{Edges adjacent to a vertex of a discrete A-net lie in a 2-plane.}
  \label{fig:planar_star}
\end{figure}

\paragraph{Equivalence of vertex and contact element description of A-nets.}
Recall, that contact elements in $\RP^3$ can be written as $(\rp x, \rp P)$ (see
Section~\ref{subsec:background}).  The two above descriptions are related as
follows: Let $v$ and~$w$ be two adjacent vertices of the quad
graph~$\mathcal{D}$, and~$\rp P(v)$ and~$\rp P(w)$ be the planes containing the
respective vertices $\rp f(v)$ and $\rp f(w)$ of the A-net.  Vertices and
planes are encoded by the contact elements $L(v) = (\rp f(v),\rp P(v))$ and
\mbox{$L(w) = (\rp f(w),\rp P(w))$} in $\plisolines$.  The line $\rp l$ supporting the
edge $(\rp f(v),\rp f(w))$ in~$\RP^3$ is exactly the intersection $l = L(v)
\cap L(w) \in \plquad$, cf.  Fig.~\ref{fig:contact_elements_anet}. We
associate~$l$ to the corresponding edge $(v,w)$ of~$\mathcal{D}$. Those lines
common to adjacent contact elements are called \emph{discrete asymptotic lines}
of the A-net. Representatives of the discrete asymptotic lines in $\plspace$
constitute the \emph{focal net} of the discrete line congruence $L$.
In the book of Bobenko and Suris~\cite{BobenkoSuris:2008:DDGBook} the definition
of focal net is given for the regular $\Z^2$-grid. Hence there exist
distinguished coordinate directions and they define multiple focal nets for
different coordinate directions.

\begin{figure}[tb]
\begin{center}
 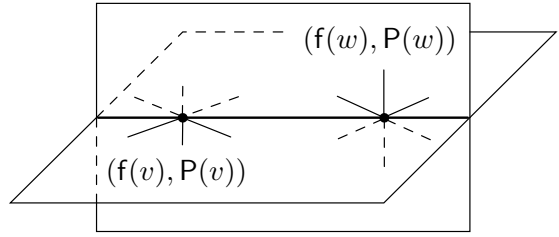 
\end{center}
\caption{Discrete asymptotic lines of an A-net are the lines shared by adjacent contact elements.}
\label{fig:contact_elements_anet}
\end{figure}

\paragraph{Affine A-nets.}
Note that, strictly speaking, in the projective setting there are no
distinguished edges, i.e., line segments connecting adjacent vertices of an
A-net. However, considering an A-net in an affine part $\R^3 \subset \RP^3$,
there is a distinguished plane at infinity, and it is natural to connect the
vertices of the A-net with finite edges as in Fig.~\ref{fig:planar_star}.  We
call such A-nets equipped with finite edges \emph{affine A-nets}.

\paragraph{Genericity assumption.}
We assume that the A-nets under consideration are generic in the following sense.
Discrete asymptotic lines are skew, except if they are adjacent to a common
vertex.  In terms of discrete congruences of isotropic lines in $\plquad$, this
means that two vertices of the focal net span an isotropic line if and only if
they are associated with edges of $\mathcal{D}$ that are adjacent to a common
vertex.  Moreover, we assume that discrete asymptotic lines meeting at a vertex
are distinct. This implies that the vertices of the focal net associated with
an elementary quadrilateral of $\mathcal{D}$, i.e., representatives in
$\plspace$ of the four discrete asymptotic lines of an elementary quadrilateral,
span a 3-space in $\plspace$ (see Fig.~\ref{fig:elem_quad_line_congruence}).

Note that if one elementary quadrilateral of an A-net in $\RP^3$ is contained
in a 2-plane, the whole net has to be contained in this plane.  In
particular the A-net is not generic.

\paragraph{Notation for discrete maps.}
A special instance of a quad-graph is (a piece of) $\Z^2$.  In this case it is
convenient to represent shifts in lattice directions by lower indices. For
$z = (m,n) \in \Z^2$ and a map $\phi$ on $\Z^2$ we write
\begin{equation*}
\phi_1(z) := \phi(m+1,n), \
\phi_{11}(z) := \phi(m+2,n), \
\phi_2(z) := \phi(m,n+1), \quad \text{etc.}
\end{equation*}
However, also for a general quad-graph $\mathcal{D}$ we use indices as
shift operators when talking about local properties, if the quad-graph
can be identified locally with $\Z^2$.

Usually we omit the argument for discrete maps and write
\begin{equation*}
  \phi = \phi (z), \quad \phi_1 = \phi_1 (z) \quad \text{etc.}
\end{equation*}

The structure of A-nets as isotropic line congruences and the notation we use
is shown in Fig.~\ref{fig:elem_quad_line_congruence}.  Upper indices are used
to denote the focal net. The lower index reflects the respective net direction. 

\begin{figure}[t]
\begin{center}
 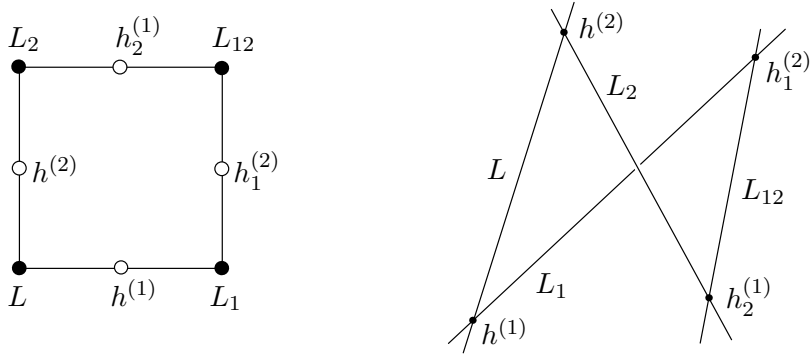 
\end{center}
\caption{Notation for a discrete congruence of lines $L$ and the corresponding focal net $h$,
combinatorially and geometrically.}
\label{fig:elem_quad_line_congruence}
\end{figure}

\subsection{Pre-hyperbolic nets}
\label{subsec:prehypnets}

Similar to the previous section, we give two descriptions of pre-hyperbolic nets,
one in~$\RP^3$ and another one in~$\plspace$ in terms of P\"ucker
geometry. Then we show that an elementary quadrilateral of an A-net admits a
1-parameter family of adapted hyperboloids. Finally, we discuss the extension
of A-nets to pre-hyperbolic nets and show in
Theorem~\ref{thm:nets_of_hyperboloids_from_anets}, that a simply connected
A-net allows for such an extension, if and only if all interior vertices are of
even degree.  In this case there is a 1-parameter family of adapted
pre-hyperbolic nets.  As the name suggests, pre-hyperbolic nets are a first
step towards hyperbolic nets. While hyperbolic nets contain hyperboloid surface
patches associated with elementary quadrilaterals, for a pre-hyperbolic net the
whole supporting hyperboloids are considered.  The motivation to introduce
pre-hyperbolic nets is two-fold. On the one hand, they provide a convenient way
to prove the consistency of the extension of A-nets to hyperbolic nets. 
On the other hand they are interesting structures in their own right,
especially in their Pl\"ucker geometric description.\footnote{In exactly the
same spirit one could introduce pre-cyclidic nets  in
\cite{BobenkoHuhnen-Venedey:2011:cyclidicNets} and use them to prove the
consistency of the propagation of Dupin cyclides}  
As we will see in Section~\ref{sec:hyperbolic_nets}, a hyperbolic net can be seen as a the restriction of a
pre-hyperbolic net.  

A generic discrete A-net with hyperboloids associated with elementary
quadrilaterals is a \emph{pre-hyperbolic net}, if
\begin{enumerate}[i)]
\item the four discrete asymptotic lines associated with edges of an elementary
  quadrilateral of the A-net are asymptotic lines of the corresponding hyperboloid, and
\item hyperboloids belonging to edge-adjacent quadrilaterals have coinciding
  contact elements along the common (discrete) asymptotic line $\rp h$, i.e., the
  hyperboloids are tangent along $\rp h$.
\end{enumerate}

\noindent We translate i) and ii) to the projective model of Pl\"ucker geometry:
\begin{enumerate}[i)]
\item The four lines $\rp h^{(1)}, \rp h^{(1)}_2$ and $\rp h^{(2)},\rp
  h^{(2)}_1$ associated with the edges of a quadrilateral of the A-net correspond
  to four points $h^{(1)}, h^{(1)}_2$ and $h^{(2)}, h^{(2)}_1$ on the Pl\"ucker
  quadric $\plquad$, cf. Fig.~\ref{fig:elem_quad_line_congruence}.  The
  hyperboloid~$\mathcal{H} \subset \RP^{3}$ associated with the quadrilateral
  corresponds to two polar planes $P^{(1)}$ and~$P^{(2)}$ in~$\plspace$ (cf.\
  Theorem~\ref{thm:hyperboloids_pluecker}). Hence the edges of the
  quadrilateral are segments of asymptotic lines of the
  hyperboloid~$\mathcal{H}$ if the points $h^{(1)}, h^{(1)}_2$ resp.\ $h^{(2)},
  h^{(2)}_1$ lie in the plane $P^{(1)}$ resp.\ $P^{(2)}$.
\item Let $\mathcal{H}$ and $\mathcal{H}_i$ be the two hyperboloids associated
	with edge-adjacent quadrilaterals with associated planes $P^{(1)},P^{(2)}$
	and $P^{(1)}_i,P^{(2)}_i$, respectively. Then $\mathcal{H}$ and
	$\mathcal{H}_i$ are tangent at the common edge $h^{(j)}_i,\ i \ne j$, (see
	Fig.~\ref{fig:edge_adjacent_line_congruence_quads}) if all contact elements
	of the two hyperboloids along this edge coincide. So let $l^{(i)} \in
	P^{(i)}\cap \plquad$ resp.\ $l^{(i)}_i \in P^{(i)}_i \cap \plquad$ be two
	asymtotic lines of $\mathcal{H}$ resp.\ $\mathcal{H}_i$ through one point on
	the common edge.  Then the contact elements $\inc{h^{(j)}_i,l^{(i)}}$ and
	$\inc{h^{(j)}_i,l^{(i)}_i}$ coincide,
	cf.~Fig.~\ref{fig:contact_element_cone}. Hence $\mathcal{H}$ and
	$\mathcal{H}_i$ are tangent along the common edge if $\inc{h^{(j)}_i,
	P^{(i)}} = \inc{h^{(j)}_i,P^{(i)}_i}$
\end{enumerate}

\begin{figure}[t]
\begin{center}
 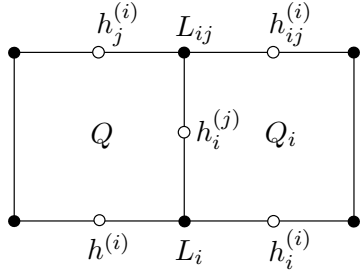 
\end{center}
\caption{Lines and contact elements of edge-adjacent quadrilaterals of an A-net.}
\label{fig:edge_adjacent_line_congruence_quads}
\end{figure}

\begin{figure}[ht]
\begin{center}
 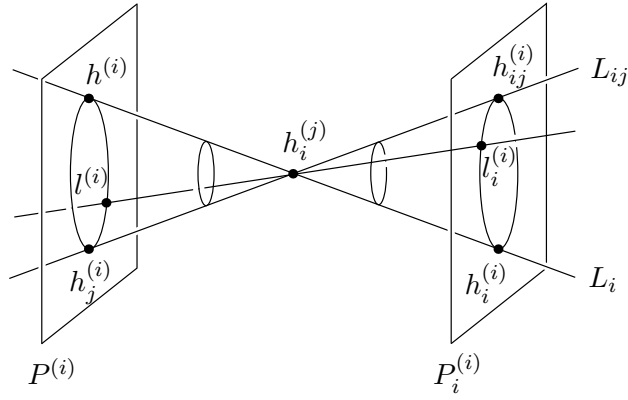 
\end{center}
\caption{Planes $P^{(i)}$ and $P^{(i)}_i$ of adjacent hyperboloids
  $\mathcal{H}$ and~$\mathcal{H}_i$ that are tangent along the
  common asymptotic line $h_i^{(j)}$ ($i \ne j$). }
\label{fig:contact_element_cone}
\end{figure}

We use this insight for the definition of pre-hyperbolic nets in terms of
Plücker geometry.

\begin{definition}[Pre-hyperbolic net]
	A generic discrete congruence of isotropic lines~$L$ in $\plquad$ defined on
	a quad-graph $\mathcal{D}$ with polar $2$-planes $P^{(1)}, P^{(2)} \in
	\hypplanes$ associated with elementary quadrilaterals is a pre-hyperbolic net, if
  \begin{enumerate}[i)]
  \item $h^{(1)},h^{(1)}_2 \in P^{(1)}$ and $h^{(2)},h^{(2)}_1 \in P^{(2)}$, and
	\item for hyperboloids associated with adjacent quadrilaterals $Q,Q_i$ of $\mathcal{D}$ as in
	Fig.~\ref{fig:edge_adjacent_line_congruence_quads},	the subspaces
	$\inc{h^{(j)}_i, P^{(i)}}$ and $\inc{h^{(j)}_i,P_i^{(i)}}$ of $\plspace$
	coincide for $i \ne j$, cf. Fig.~\ref{fig:contact_element_cone}.
  \end{enumerate}
	\label{def:pre_hyperbolic_net}
\end{definition}

\begin{remark}
Property ii) in Definition~\ref{def:pre_hyperbolic_net} implies that
corresponding reguli of adjacent hyperboloids have the same orientation. 
In Fig.~\ref{fig:contact_element_cone}, if the signature of $P^{(i)}$
is $({+}{+}{-})$, then the 3-space $\inc{h_i^{(j)},P^{(i)}}$ is of
(degenerate) signature $({+}{+}{-}{0})$.  Since $P_i^{(i)}$ does not contain
$h_i^{(j)}$, the plane $P_i^{(i)}$ is also of signature $({+}{+}{-})$.  The
fact that the two reguli of a single hyperboloid are of complementary signature
in turn implies, that $P^{(j)}$ and $P_i^{(j)}$ also have the same signature.
\label{rem:propagation_preserves_orientation}
\end{remark}

\paragraph{Extending discrete A-nets to pre-hyperbolic nets.}
For an elementary quadrilateral $Q=(L,L_1,L_{12},L_2)$ of an A-net we define
\begin{eqnarray*}
  \inc{Q} &:=& \inc{L,L_1,L_{12},L_2},\\
  H[Q] &:=& \pol{L,L_1,L_{12},L_2} = \pol{Q}.
\end{eqnarray*}
Usually we write $H,H_i$ for $H[Q], H[Q_i]$, etc.

\begin{proposition}[]
Let $Q$ be an elementary quadrilateral of a generic A-net that is defined on a
quad-graph $\mathcal{D}$. Then:
\begin{enumerate}[i)]
\item The space $\inc{Q} \subset \plspace$ is 3-dimensional and of signature
  $({+}{+}{-}{-})$.  The space $H[Q] \subset \plspace$ is a projective line of signature
  $(+-)$.
\item The lines $ H = \pol{Q}$ associated with the elementary quadrilaterals $Q$
  of the A-net constitute a discrete line congruence on the dual graph of $\mathcal{D}$.
\end{enumerate}
\label{prop:properties_of_non_degenerate_anet_quads}
\end{proposition}

\begin{proof}
  \emph{i)} An elementary quadrilateral of a generic A-net is a
  non-planar quadrilateral in $\RP^3$. According to Remark~\ref{rem:tetrahedron},
  the pairs $h^{(2)},h^{(2)}_1$ and $h^{(1)},h^{(1)}_2$, as well as the
  diagonals $g_1,g_2$ each span lines of signature~$(+-)$ in $\plspace$. So
  $H[Q] = \inc{g_1,g_2}$ has signature $(+-)$ and thus its polar $\inc{Q}$
  has signature $({+}{+}{-}{-})$.

  \emph{ii)} It remains to show that for edge-adjacent quadrilaterals $Q$ and $Q_i$ the
  lines $H$ and $H_i$ intersect.  This follows from the observation, that $L_i$
	and $L_{ij}$, $i \ne j$, are both contained in each of the projective spaces
	$\inc{Q}$ and $\inc{Q_i}$, and hence $H, H_i \subset \pol{L_i,L_{ij}}$, cf.
  Fig.~\ref{fig:edge_adjacent_line_congruence_quads}.  Since $\pol{L_i,L_{ij}}$
  is a projective plane, the lines $H$ and $H_i$ intersect.
\end{proof}

Each of the elementary quadrilaterals admits a 1-parameter family of adapted
hyperboloids as indicated in Fig.~\ref{fig:family}. These hyperboloid can be described 
in terms of Pl\"ucker geometry as explained in the following Lemma~\ref{lem:hyperboloids_from_anets}.

\begin{lemma}
  For an elementary quadrilateral $Q$ of a generic A-net there is a
  1-parameter family of hyperboloids containing all four discrete asymptotic
  lines.  Each such hyperboloid is uniquely determined by the choice of two
  labeled points $q^{(1)},q^{(2)} \in H[Q]$ that are polar (see Fig.~\ref{fig:family} left).
\label{lem:hyperboloids_from_anets}
\end{lemma}

\begin{proof}
For each non-isotropic point $q \in H = H[Q]$ there is a unique polar point on
$H$, due to the signature $(+-)$ of $H$.  Each pair of polar points
$q^{(1)},q^{(2)} \in H\setminus \plquad$ determines a pair of polar planes
$P^{(i)} := \inc{h^{(i)},h^{(i)}_j,q^{(i)}}$ in~$\hypplanes$. These planes
define a unique hyperboloid with prescribed asymptotic lines, cf.
Fig.~\ref{fig:poles_in_polar_planes}.  Conversely, any hyperboloid
$\mathcal{H}$ that contains the given discrete asymptotic lines can be
constructed in this way, since associated polar planes $P^{(i)} \in \hypplanes$
intersect $H$ in polar points $q^{(i)}$.
\end{proof}

\begin{figure}[hbt]
\begin{center}
 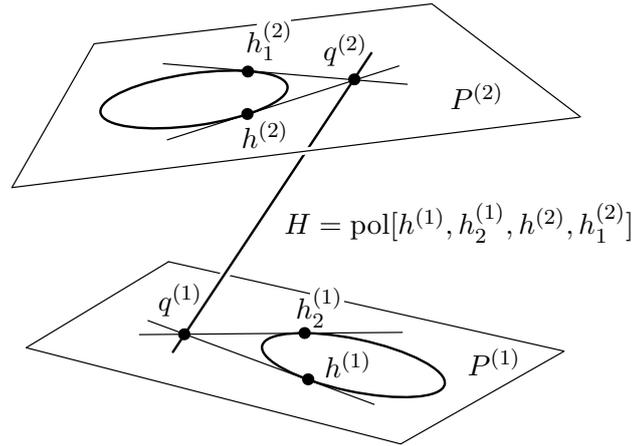 
\end{center}
\caption{Configuration in $\plspace$ which describes the extension of an
elementary quadrilateral of a discrete A-net to a hyperboloid.  Given
representatives $h^{(1)},h_2^{(1)},h^{(2)},h_1^{(2)} \in \plquad$ of the
discrete asymptotic lines, the extension is performed by choosing polar points
$q^{(1)},q^{(2)} \in H$.}
\label{fig:poles_in_polar_planes}
\end{figure}

\begin{remark}
\label{rem:specialQ}
  There are two special choices for $q \in H$, namely the isotropic points in $H
  \cap \plquad$. In this case, $\rp q$ is one of the two diagonals of the
  non-planar quadrilateral in $\RP^3$.  Furthermore, $q$ is self polar and does
  not yield a decomposition of $\plspace$ into two disjoint planes in
  $\hypplanes$ as required by Theorem~\ref{thm:hyperboloids_pluecker}.  Hence
  there is no hyperboloid associated with $q = q^{(1)} = q^{(2)}$. Instead this choice corresponds to
  one of the two limiting cases of hyperboloids $\mathcal{H}$ for which $\rp
  h^{(1)},\rp h_2^{(1)},\rp h^{(2)},\rp h_1^{(2)}$ are asymptotic lines. The
  hyperboloid degenerates to two planes in $\RP^3$ that contain those lines and
  intersect in the diagonal~$\rp q$ (see Fig.~\ref{fig:family} right).
  Considering $P^{(1)} = \inc{h^{(1)},h_2^{(1)}, q}$ and $P^{(2)} =
  \inc{h^{(2)},h_1^{(2)}, q}$ as for hyperboloids, each intersection $P^{(i)}
  \cap \plquad$ with the Pl\"ucker quadric gives a degenerate conic.  These
  degenerate conics each consist of two isotropic lines, i.e., two contact
  elements. The planes of these contact elements are exactly the planes in
  $\RP^3$ intersecting in $\rp q$, and the points of the contact elements are
  the vertices of the quadrilateral that are adjacent to $\rp q$.
\end{remark}

\begin{figure}[b]
  \begin{center}
    \includegraphics[width=.35\linewidth]{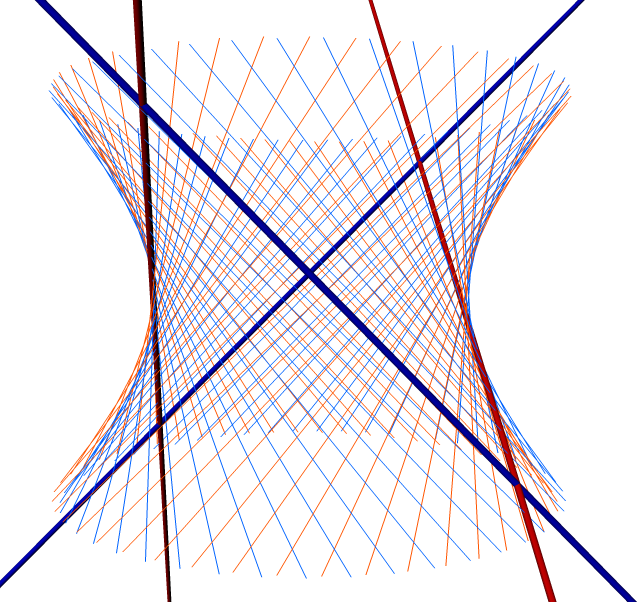} \hspace{1cm}
    \includegraphics[width=.35\linewidth]{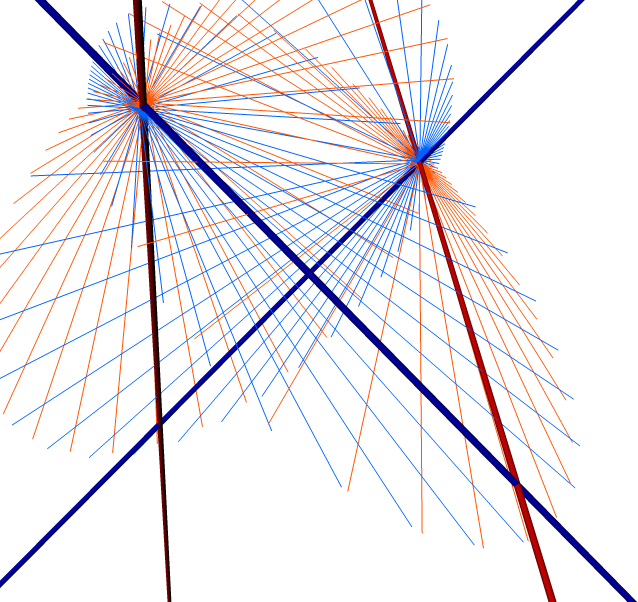} \\[1cm] 
    \includegraphics[width=.35\linewidth]{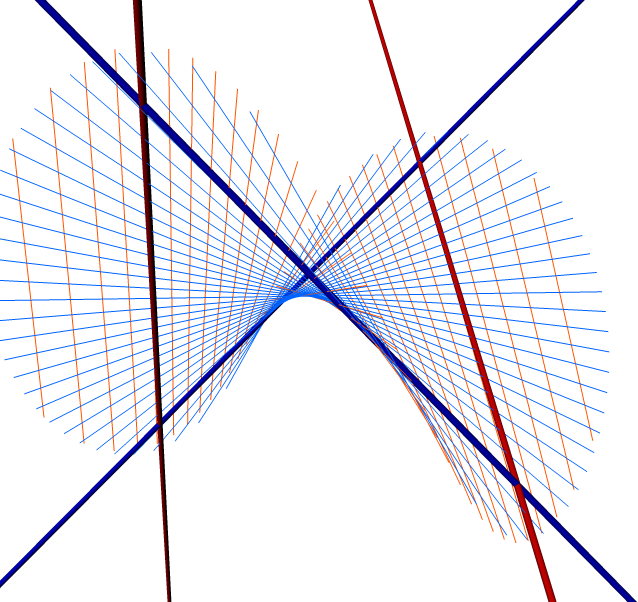} \hspace{1cm}
    \includegraphics[width=.35\linewidth]{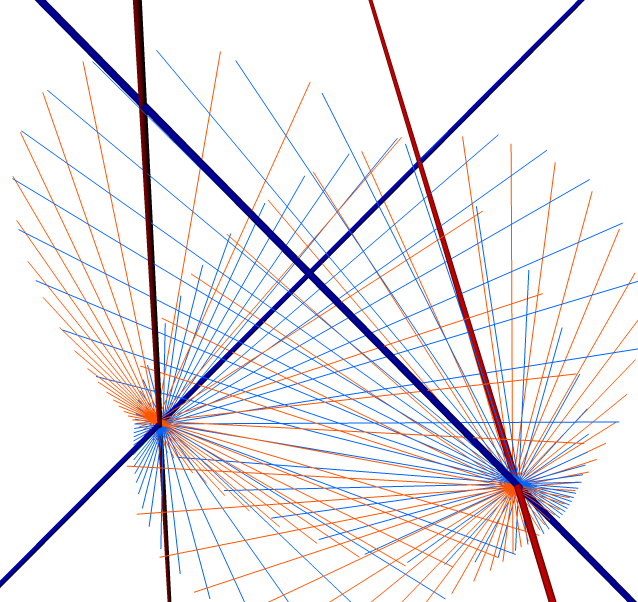}
  \end{center}
  \caption{The 1-parameter family of adapted hyperboloids supporting an elementary
  quadrilateral of an A-net: Two non-degenerate hyperboloids (left) and the two degenerate
  cases (right).} 
  \label{fig:family}
\end{figure}

The $C^1$-condition for adjacent hyperboloids in a pre-hyperbolic net determines
the whole net if one initial hyperboloid is given. In particular, there exists a
1-parameter family of pre-hyperbolic nets for a generic discrete A-net.  To
prove consistency of the evolution, we first describe the propagation of
a hyperboloid from one quadrilateral to an adjacent one as a projection
between~$H$ and $H_i$.

\begin{lemma}[]
  For a pre-hyperbolic net let $Q$ and $Q_i$ be edge-adjacent quadrilaterals of
  the underlying A-net.  Then the line $H_i = H[Q_i]$ is the projection of
  $H=H[Q]$ onto $\pol{h_{ii}^{(j)}}$ through the point $h_i^{(j)}, i \ne j$.  In
  particular, the points $q^{(1)}$ and $q^{(2)}$ are mapped onto $q_i^{(1)}$ and
  $q_i^{(2)}$, cf. Fig.~\ref{fig:projected_H}.
  \label{lem:qs_for_adjacent_hyperboloids}
\end{lemma}
\begin{figure}[ht]
  \begin{center}
     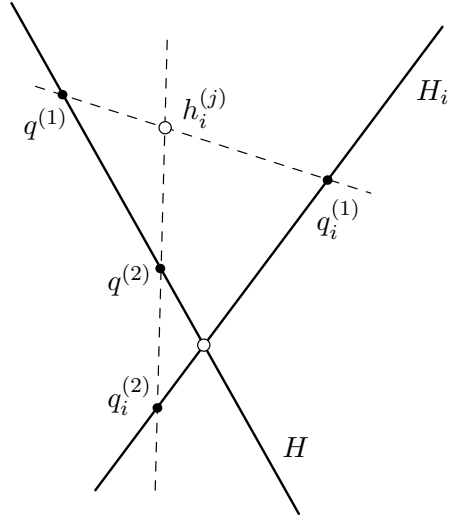 
  \end{center}
  \caption{For pre-hyperbolic nets, $H_i$ is the projection of $H$ onto
  $\pol{h_{ii}^{(j)}}$ through $h_i^{(j)}$. In particular $(q^{(1)},q^{(2)})$
  is mapped to $(q_i^{(1)},q_i^{(2)})$ and $h_i^{(j)}$ lies in the plane
  $\inc{H,H_i}$.}
  \label{fig:projected_H}
\end{figure}
\begin{proof}
For generic A-nets, the point $h_i^{(j)} \in \plquad$, $i \ne j$, is not
contained in $\pol{h_{ii}^{(j)}}$ and the projection \begin{eqnarray}
\label{eq:q_projection} \tau : \plspace \setminus \{ h_i^{(j)} \} &\to&
\pol{h_{ii}^{(j)}}\\ \nonumber q &\mapsto& \inc{q,h_i^{(j)}} \cap
\pol{h_{ii}^{(j)}} \end{eqnarray} through $h_i^{(j)}$ is well defined.  It is
crucial that $\tau$ maps $\pol{h_i^{(j)}}$ onto $\pol{h_i^{(j)}} \cap
\pol{h_{ii}^{(j)}}$.  The reason is that $h_i^{(j)}$ is isotropic, and hence
for all $q \in \pol{h_i^{(j)}}$ the projection line $\inc{q,h_i^{(j)}}$ is
completely contained in $\pol{h_i^{(j)}}$.  Since for a pre-hyperbolic net
$\inc{P^{(i)},h_i^{(j)}} = \inc{P_i^{(i)},h_i^{(j)}}$, the space
$P^{(i)} \subset \pol{h_i^{(j)}}$ is mapped to $P_i^{(i)} \subset
\pol{h_i^{(j)}} \cap \pol{h_{ii}^{(j)}}$, cf.
Fig.~\ref{fig:contact_element_cone}.
This means
\begin{equation}
\tau (P^{(i)}) = P_i^{(i)}
\label{eq:hypplanes_are_mapped_onto}
\end{equation}
and in particular
\begin{equation}
\tau (h^{(i)}) = h_i^{(i)} \quad \text{and} \quad \tau (h_j^{(i)}) = h_{ij}^{(i)}.
\label{eq:pluecker_lines_are_mapped_onto}
\end{equation}
The claim of the Lemma
\begin{equation}
\tau (H) = H_i
\label{eq:plusplus_Lines_are_mapped_onto}
\end{equation}
holds, since $\tau$ preserves polarity: 
Let $q,q' \in \pol{h_i^{(j)}}$ with $\langle \pc q , \pc q' \rangle = 0$.
Thus for all $\alpha,\alpha',\beta,\beta' \in \R$ we have
\begin{equation*}
\langle \alpha \, \pc q + \beta \, \pc h_i^{(j)} , \alpha' \, \pc q' + \beta' \, \pc h_i^{(j)} \rangle = 0
\quad \implies \quad
\tau(q), \tau(q') \text{ are polar}.
\end{equation*}
With Eq.~\eqref{eq:pluecker_lines_are_mapped_onto} and definition of $\tau$
this implies that $\pol{h_i^{(j)},h^{(i)},h_j^{(i)}} \supset H$ is mapped onto
$\pol{h_i^{(j)},h_i^{(i)},h_{ij}^{(i)},h_{ii}^{(j)}} = H_i$.  Together with
Eq.~\eqref{eq:hypplanes_are_mapped_onto} and
Eq.~\eqref{eq:pluecker_lines_are_mapped_onto}, preservation of polarity also
implies that $q^{(i)}$ (the unique point in $P^{(i)}$ that is polar to
$h^{(i)}$ and $h_j^{(i)}$, cf.  Fig.~\ref{fig:poles_in_polar_planes}) is mapped
to $q_i^{(i)}$ (the unique point in $P_i^{(i)}$ that is polar to $h_i^{(i)}$
and $h_{ij}^{(i)}$).  Finally the preservation of polarity toghether with
Eq.~\eqref{eq:plusplus_Lines_are_mapped_onto} gives $\tau(q^{(j)}) =
q_i^{(j)}$.
\end{proof}

We are now ready to proof the main theorem about pre-hyperbolic nets.

\begin{theorem}
  For a given simply connected generic A-net with all interior vertices of even
  degree, there exists a 1-parameter family of pre-hyperbolic nets.  Each such
  net is determined by the choice of one initial hyperboloid.
  \label{thm:nets_of_hyperboloids_from_anets}
\end{theorem}

\begin{proof}
Lemma~\ref{lem:hyperboloids_from_anets} shows that for each elementary
quadrilateral of an A-net, there is a 1-parameter family of compatible
hyperboloids.  Hence we have a 1-parameter freedom for the choice of one
initial hyperboloid.  Lemma~\ref{lem:qs_for_adjacent_hyperboloids} explains how
this hyperboloid determines its four neighboring hyperboloids associated with
edge-adjacent quadrilaterals. It remains to show that this evolution is
consistent, i.e., that the propagation of hyperboloids is independent of the
particular chosen path. Since the quad-graph is simply connected, there are no
non-trivial cycles. Thus we have to consider consistency of the evolution
around single inner vertices only.

\begin{figure}[tb] 
  \begin{center}
     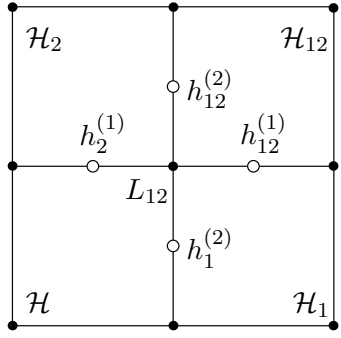 
  \end{center}
  \caption{All four projection centers $h_i^{(j)}$ lie on the line $L_{12}$.} 
  \label{fig:four_quad_projection}
\end{figure} 

First, we consider a regular vertex of degree four as in
Fig.~\ref{fig:four_quad_projection}. We start with the two points $q^{(1)}$ and
$q^{(2)}$ on $H = \pol{Q}$, which determine the hyperboloid $\mathcal{H}$
associated with the lower left quadrilateral. By
Lemma~\ref{lem:qs_for_adjacent_hyperboloids}, the projections of
$(q^{(1)},q^{(2)})$ onto $H_1$ and $H_2$ according to
Eq.~\eqref{eq:q_projection} yield the adjacent hyperboloids $\mathcal{H}_1$ and
$\mathcal{H}_2$, respectively.  Consistency for all initial data means that the map used for
propagation of the pair $(q^{(1)},q^{(2)})$ around the vertex is the identity.
More precisely, we have to check whether the evolution satisfies two properties.
Firstly, according to the notation of Lemma~\ref{lem:qs_for_adjacent_hyperboloids},
it has to hold
\begin{equation} 
  \tau_2^{(1)} \circ \tau_{12}^{(2)} \circ \tau_{12}^{(1)} \circ \tau_1^{(2)} = \id
  \quad \Leftrightarrow \quad
  \tau_{12}^{(1)} \circ \tau_1^{(2)} = \tau_{12}^{(2)} \circ \tau_2^{(1)},
  \label{eq:consistency_condition_for_projections} 
\end{equation} 
where lower and upper indices correspond to the center of projection.  Here
$\tau_2^{(1)}, \tau_{12}^{(2)}, \tau_{12}^{(1)}, \tau_1^{(2)}$ are understood
as involutions, for example the projection $\tau_1^{(2)} : H \to H_1$ through
$h_1^{(2)}$ as map in $\plspace$ is inverted by exchanging domain and co-domain
(cf.  Fig.~\ref{fig:projected_H}).  Moreover, we have to prove that the
labeling of $q$'s is preserved. 

To prove Eq.~\eqref{eq:consistency_condition_for_projections}, first note that
by definition of the lines $H,H_1,H_{12},H_2$, the space $V =
\inc{H,H_1,H_{12},H_2}$ is contained in the 3-space $\pol{L_{12}} \supset
L_{12}$.  Moreover, $\inc{H,H_1} \cap L_{12} = h_1^{(2)}$ and similarly for
the other $H$-lines, cf.~Fig.~\ref{fig:projected_H}.  As the four intersection
points $h_1^{(2)},h_{12}^{(1)},h_{12}^{(2)},h_2^{(1)} \in L_{12}$ are distinct,
the space $V$ has to be at least 3-dimensional and hence $V = \pol{L_{12}}$.
As a consequence we can restrict ourselves to the projective 3-space $V$
and understand the projections $\tau_2^{(1)}, \tau_{12}^{(2)}, \tau_{12}^{(1)},
\tau_1^{(2)}$ as one and the same projection $\tau$ through $L_{12}$, just
acting on different (co-)domains.  This proves
Eq.~\eqref{eq:consistency_condition_for_projections}.

So far, exactly the same argumentation holds for arbitrary degree of the central vertex
greater than 4 and the corresponding identities analogous to
\eqref{eq:consistency_condition_for_projections}. However, in contrast to that, the
labeling of $q$'s is only preserved for even vertex degree: Given a point $q \in H =
\pol{Q}$ one can understand the labeling $q \to q^{(i)}$ as association of $q$
to an edge that is adjacent to the central vertex, say $v$, since the two edges of $Q$
adjacent to $v$ are natural representatives of the two reguli of any admissible
hyperboloid.  Now if $q$ is propagated $q \mapsto q_i = \tau_i^{(j)}(q)$, then
the labeling, i.e., allocation of $q$ to an edge adjacent to $v$, has to be
propagated accordingly (cf. Fig.~\ref{fig:propagation_of_q_labeling}). It
follows, that the labeling of $q$'s is propagated consistently, if and only if
$v$ is of even degree.
\end{proof}

\begin{figure}[h]
  \begin{center} 
     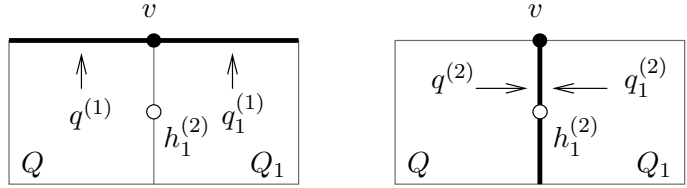  
  \end{center}
  \caption{Propagation of the labeling of $q$'s with respect to the map $\tau_1^{(2)}$,
  represented by edges adjacent to the central vertex.}
  \label{fig:propagation_of_q_labeling}
\end{figure}

\begin{remark}
A simply connected quad-graph is either a disc or a sphere, in particular it is
orientable.
Moreover, if all interior vertices are of even
degree, it has to be a disc.  This follows from the Euler characteristic,
which yields that a strongly regular cell-decomposition of a sphere
necessarily contains either triangles or vertices of degree $3$.  One can also
show this way, that the considered quad-graphs do not contain self-intersecting
or closed quadrilateral strips (see Definition~\ref{def:stripOrientation}).
\end{remark}

\section{Hyperbolic nets}
\label{sec:hyperbolic_nets}

This section aims at the extension of affine A-nets to piecewise smooth
$C^1$-surfaces, called \emph{hyperbolic nets}. An A-net in $\RP^3$ is turned
into an affine A-net by choosing an affine chart of $\RP^3$ and equipping the
A-net with finite edges connecting adjacent vertices, see
Section~\ref{subsec:anets}.  The elementary quadrilaterals of affine A-nets are
therefore skew quadrilaterals in $\R^3$. Extending an affine A-net to a
hyperbolic net amounts to fitting hyperboloid surface patches into the skew
quadrilaterals in such a way that the connection between edge-adjacent patches are
differentiable. Related to the propagation of hyperboloids in pre-hyperbolic
nets, a surface patch glued to a skew quadrilateral determines its edge-adjacent
neighbors via the $C^1$-property.  Starting with one initial
hyperboloid patch, the consistency of this propagation is derived from the
consistency of the propagation of hyperboloids for pre-hyperbolic nets.  It
turns out that not all affine A-nets that allow for an extension to pre-hyperbolic
nets can also be extended to hyperbolic nets. The additional characteristic
property is shown to be ``equi-twist'' of quadrilateral
strips.

We discuss the orientation of reguli and introduce the related twist of skew
quadrilaterals in Section~\ref{subsec:orientation_and_twist}.  In
Section~\ref{subsec:hyperbolic_nets}, we first investigate hyperboloid surface
patches focussing on patches bounded by a prescribed skew quadrilateral.
Then we introduce hyperbolic nets and show how to construct hyperbolic nets
from affine A-nets in Theorem~\ref{thm:extension}. Finally,
some comments are made on the computer implementation of our results.

\subsection{Orientation of reguli and twist of skew quadrilaterals}
\label{subsec:orientation_and_twist}

Any three skew lines in $\RP^3$ span a regulus.  The corresponding plane in
$\hypplanes$ is either of signature $({+}{+}{-})$ or of signature
$({+}{-}{-})$, where the different signatures correspond to the two possible
orientations of reguli.  Hence an orientation in~$\RP^3$ is defined for three
skew lines.  A skew quadrilateral in $\RP^3$ consists of four points in general
position connected by line segments (edges).  The notion of orientation for reguli induces a natural
orientation for a pair of opposite edges of a skew quadrilateral. We call this
orientation the \emph{twist} of an opposite edge pair.

\paragraph{Orientation of reguli.}
Consider a hyperboloid $\mathcal{H}$ in~$\RP^3$ that is described by two polar
planes~$P^{(1)}$ and~$P^{(2)}$ in $\plspace$, which are spanned by
representatives of arbitrary three lines of the corresponding reguli of
$\mathcal{H}$.
Assume $P^{(1)}$ is of signature $({+}{+}{-})$ and $P^{(2)}$ is of signature
$({+}{-}{-})$.  The signature of a plane $P$ is the signature of the Plücker
product restricted to the 3-space $\pc P$ of homogeneous coordinates. With
respect to a particular basis, the restriction $\langle \cdot , \cdot \rangle
\big|_{\pc P}$ is represented by a symmetric $3 \times 3$ matrix $M_P$, cf.
Appendix~\ref{appendix}.  Due to Sylvester's law of inertia, the signature of
$M_P$ is independent of the chosen basis and coincides with the signature of
$\langle \cdot , \cdot \rangle \big|_{\pc P}$.  As the determinant of a matrix
is the product of its eigenvalues, we have
\begin{equation*}
\det(M_{P^{(1)}}) < 0 \quad \text{and} \quad \det(M_{P^{(2)}}) > 0.
\end{equation*}
According to the sign of the determinant, the regulus $P^{(1)} \cap \plquad$ is
said to be of negative orientation, while $P^{(2)} \cap \plquad$ is
said to be of positive orientation.  For the calculation of the orientation
of a regulus (or, equivalently, the orientation of three skew lines)
in coordinates, see Appendix~\ref{appendix}.

\begin{remark}
There are different choices of Pl\"ucker line coordinates, and the orientation
of a regulus (of three skew lines) depends on the particular choice.  In a fixed
affine chart, one can identify the two possible orientations of reguli with
clockwise and counter-clockwise screw motions in space.  However, this
interpretation depends on the choice of the affine chart -- while for fixed
Pl\"ucker coordinates the orientation of a regulus is well defined, in
different affine charts the same regulus may appear rotating either clockwise or
counter-clockwise.
\label{rem:screw_motion}
\end{remark}

Now recall the extension of an elementary quadrilateral $Q$ of an A-net to a
hyperboloid, as discussed in Section~\ref{subsec:prehypnets}.  By
Lemma~\ref{lem:hyperboloids_from_anets} there exists a 1-parameter family of
different hyperboloids, whose asymptotic lines contain the discrete asymptotic
lines of the A-net (see Fig.~\ref{fig:family}). The 1-parameter family of
hyperboloids corresponds to the 1-parameter family of polar points $q, q^* \in
H[Q] = \pol{h^{(1)}, h^{(1)}_2,h^{(2)}, h^{(2)}_1}$.  We ignore the degenerate
case $q = q^*$, which is equivalent to $\langle \pc q, \pc q \rangle = 0$ (see
Remark~\ref{rem:specialQ}).  In the generic case $\langle \pc q, \pc q \rangle$
and $\langle \pc{q}^*, \pc{q}^* \rangle$ have opposite signs. Such a generic pair
of polar points yields two different hyperboloids, depending on the labeling of
the points, i.e., either $q^{(1)} = q$ and $q^{(2)}=q^*$ or the other way
around.  The essential difference between the labelings is the following:
Changing the labeling also changes the signatures of $P^{(1)} = \inc{h^{(1)},
h^{(1)}_2, q^{(1)}}$ and $P^{(2)} = \inc{h^{(2)}, h^{(2)}_1, q^{(2)}}$, i.e.,
the orientation of the hyperbolic families of lines $P^{(1)} \cap \plquad$ and
$P^{(2)} \cap \plquad$.

As mentioned in Remark~\ref{rem:screw_motion}, in a fixed affine chart the
orientation of a regulus can be described as a clockwise or counter-clockwise
screw motion. We will now explain this geometric interpretation in the context
of the extension of a quadrilateral of an A-net to an adapted hyperboloid.
Extending the pair of lines $\rp h^{(1)}$ and $\rp h_2^{(1)}$ to a regulus of
an adapted hyperboloid corresponds to
choosing a third line~$\rp h$ that is skew to $\rp h^{(1)}$ and $\rp
h_2^{(1)}$, but that intersects $\rp h^{(2)}$ and $\rp h_1^{(2)}$, cf.
Fig.~\ref{fig:regulus_orientation}. Referring to the affine chart used in
Fig.~\ref{fig:regulus_orientation}, the regulus $\mathcal{H}^{(1)} =
\inc{h^{(1)}, h, h_2^{(1)}} \cap \plquad$ is oriented clockwise.  This reflects
how the lines of $\mathcal{H}^{(1)}$ twist around any line of the complementary
regulus $\mathcal{H}^{(2)} = \pol{h^{(1)}, h, h_2^{(1)}} \cap \plquad$.  In
particular, the twist direction is independent of whether you choose $\rp
h^{(2)}$, or $\rp h_1^{(2)}$ as reference. It is also independent of the
direction of traversal of the reference line.  (Keep in mind that every two
lines of one regulus of a hyperboloid are skew.  Hence for each regulus there
is a unique screw motion that identifies three lines of a regulus.)
If one considers the regulus
$\tilde{\mathcal{H}}^{(1)} = \inc{h^{(1)}, \tilde h, h_2^{(1)}} \cap \plquad$,
that is spanned by the skew lines $\rp h^{(1)}, \rp{\tilde h}$, and~$\rp
h_2^{(1)}$ instead, then $\tilde{\mathcal{H}}^{(1)}$ is oriented
counter-clockwise. 

\begin{figure}[htb]
  \begin{center}
     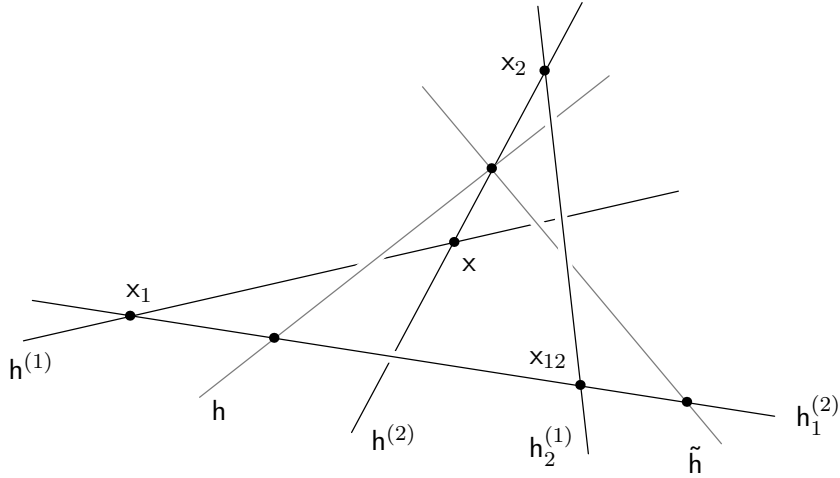 
  \end{center}
  \caption{
    With respect to a fixed affine chart, the orientation of a regulus can be 
    referred to as a clockwise or counter-clockwise screw motion.
    The regulus spanned by $\rp h^{(1)}, \rp h, \rp h_2^{(1)}$ is oriented
    clockwise, while the regulus spanned by $\rp h^{(1)}, \rp{\tilde h}, \rp
    h_2^{(1)}$ is oriented counter-clockwise.}
  \label{fig:regulus_orientation}
\end{figure}

\paragraph{Twist of skew quadrilaterals.}
Any skew quadrilateral in $\RP^3$ can be seen as a skew quadrilateral in
$\R^3$, i.e., there is an affine chart such that all four edges of the
quadrilateral are finite. Accordingly, consider Fig.~\ref{fig:regulus_orientation} in
affine $\R^3$, such that there is a unique skew quadrilateral with vertices $\rp
x$, $\rp x_1$, $\rp x_2$, and $\rp x_{12}$ and finite edges connecting them.
The \emph{twist of the opposite edge pair} $(\rp x,\rp x_1)$ and $(\rp x_2,\rp
x_{12})$ is defined as the orientation of any regulus that is spanned by $\rp
h^{(1)}, \rp h^{(1)}_2$, and a third line $\rp h$, where $\rp h$ is skew to
$\rp h^{(1)}, \rp h^{(1)}_2$ but intersects the complementary lines $\rp
h^{(2)}$ and~$\rp h^{(2)}_1$ in their finite segments, i.e., in the edges of
$Q$.  So the twist of an opposite edge pair can be either positive, or
negative.  Intuitively it is clear, that the twist does not depend on the
particular choice of the line $\rp h$; a proof using coordinates can be found in Appendix~\ref{appendix}. The other
edge pair, $(\rp x,\rp x_2)$ and $(\rp x_1,\rp x_{12})$, automatically has the
opposite twist.

\subsection{Hyperbolic nets}
\label{subsec:hyperbolic_nets}

In Section~\ref{sec:preHypNets} it was shown that simply connected A-nets with
all interior vertices of even degree can be extended to pre-hyperbolic nets.
Now we will characterize affine A-nets that can be extended to piecewise smooth
$C^1$-surfaces parametrized along asymptotic lines, by attaching hyperboloid
patches to the skew quadrilaterals. The resulting surfaces are called
hyperbolic nets. We start with the formal definition of hyperboloid patches and
the extension of skew quadrilaterals to such patches, before defining
hyperbolic nets similar to pre-hyperbolic nets.

\begin{definition}[Hyperboloid patch]
A \emph{hyperboloid patch} is a (parametrized) surface patch, obtained by
restricting a (global) asymptotic line parametrization of a hyperboloid to a
closed rectangle.
\label{def:hyperboloid_patch}
\end{definition}

Geometrically, a hyperboloid patch is a piece of a hyperboloid cut out along
four asymptotic lines, two per regulus, cf.  Fig.~\ref{fig:hyperboloid_patch}.
Conversely, four asymptotic lines of a hyperboloid cut the hyperboloid into
four patches, since a hyperboloid in $\RP^3$ is a torus,
topologically.\footnote{In our setting a natural way to obtain a homeomorphism
between $S^1 \times S^1$ and the contact elements of a hyperboloid is, to
parametrize each of the corresponding hyperbolic families of lines over $S^1$
and use the structure explained in Fig.~\ref{fig:hyperbolic_planes}.}
If we choose a plane at infinity, then this plane intersects some, or all of
the patches. More precisely, in the affine setting a hyperboloid is cut by four
asymptotic lines into either four infinite patches or into three infinite patches
and one finite patch.

\begin{figure}[htb]
  \begin{center}
    \includegraphics[scale=.13]{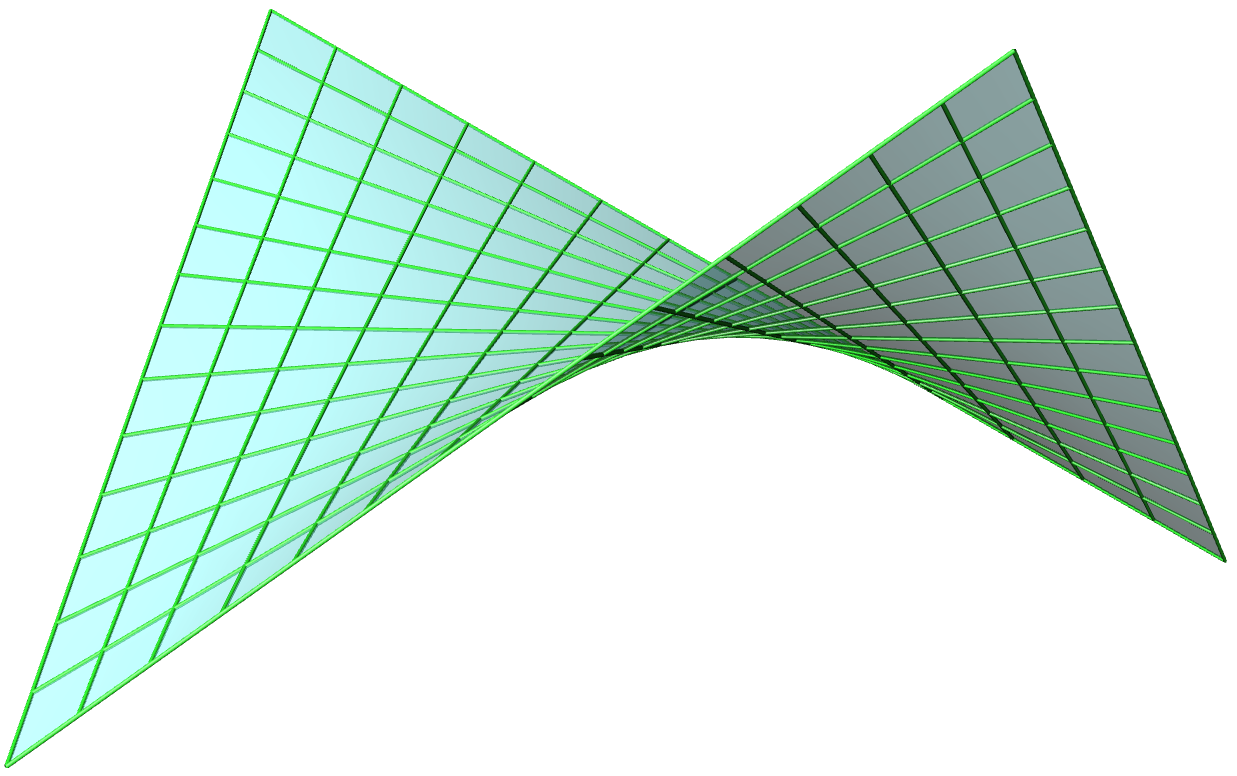}
    \hspace{25pt}
     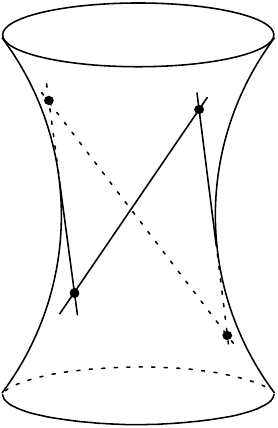 
  \end{center}
  \caption{Left: A finite hyperboloid patch.
	Right: A finite skew quadrilateral on a hyperboloid that does not bound a hyperboloid patch.}
  \label{fig:hyperboloid_patch}
\end{figure}

Using the results of Section~\ref{subsec:orientation_and_twist}, we will now
discuss the extension of skew quadrilaterals to hyperboloid patches.
Lemma~\ref{lem:hyperboloids_from_anets} tells us, that for a skew quadrilateral
$Q$ there exists a 1-parameter family of adapted hyperboloids.  The crucial
observation is, that such a hyperboloid can be restricted to a hyperboloid
patch that is bounded by $Q$, if and only if the twists of the opposite edge
pairs of $Q$ coincide with the corresponding orientations of the reguli of
$\mathcal{H}$, cf. Fig.~\ref{fig:hyperboloid_patch}. The reason for that is,
that there is a patch bounded by $Q$, if and only if asymptotic lines of
$\mathcal{H}$ that intersect one edge of $Q$ also intersect the opposite edge
of $Q$ (recall the definition of the twist of opposite edge pairs and consider
Fig.~\ref{fig:regulus_orientation}). If the twist of one opposite edge pair
coincides with the orientation of the supporting regulus, this is obviously
also the case for the other edge pair. If the edges of $Q$ are finite, then any
patch bounded by $Q$ is also finite, cf.  Fig.~\ref{fig:pathesAdaptedToQuad}.

\begin{figure}[tbh]
  \begin{center}
    \includegraphics[width=.30\linewidth]{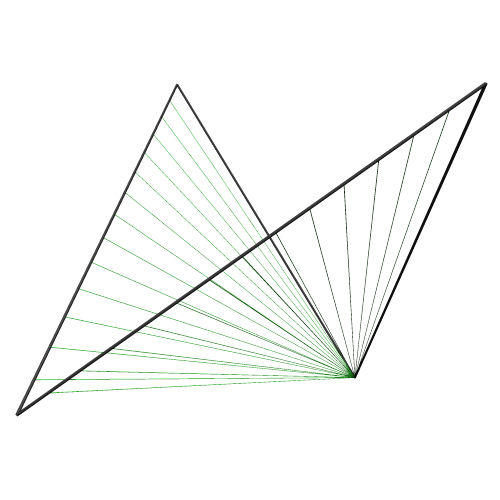}
    \includegraphics[width=.30\linewidth]{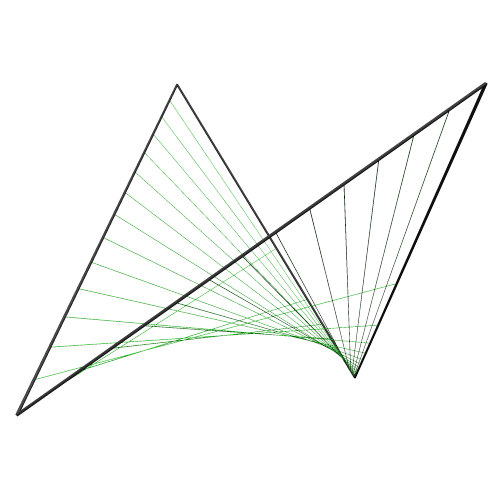}
    \includegraphics[width=.30\linewidth]{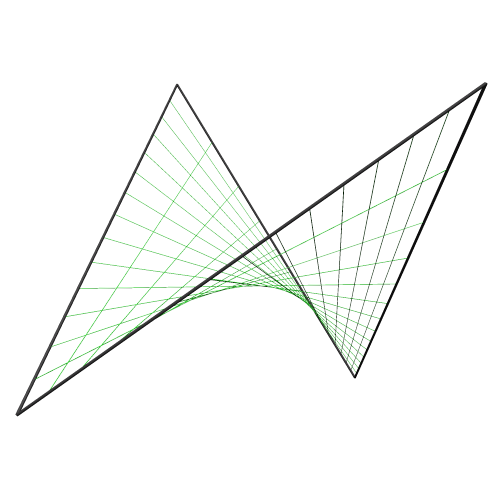}
    \includegraphics[width=.30\linewidth]{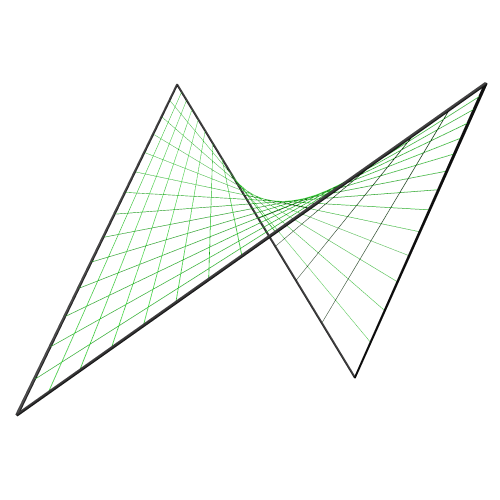}
    \includegraphics[width=.30\linewidth]{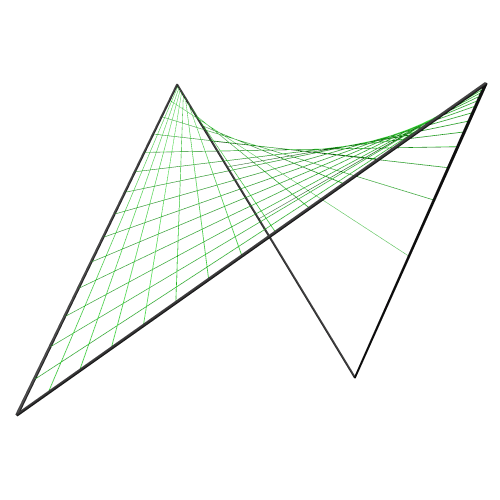}
    \includegraphics[width=.30\linewidth]{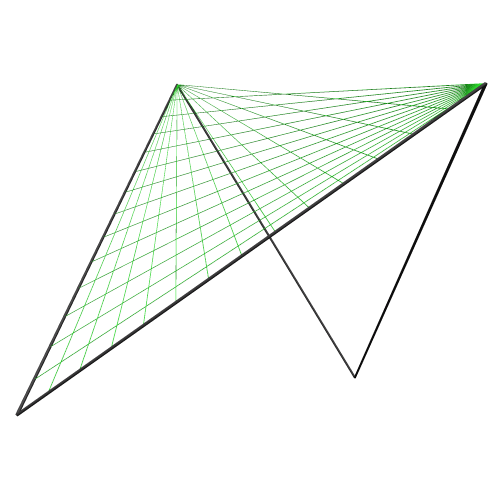}
  \end{center}
  \caption{Hyperboloid patches adapted to a skew quadrilateral.}
  \label{fig:pathesAdaptedToQuad}
\end{figure}

We summarize the above in the following Lemma~\ref{lem:adaptedPatch}

\begin{lemma}
  Let~$\mathcal{H}$ be a hyperboloid and $Q$ be a skew quadrilateral on
  $\mathcal{H}$. The hyperboloid can be restricted to a hyperboloid patch that is 
  bounded by~$Q$, if and only if the twist of each opposite edge pair of $Q$
  coincides with the orientation of the regulus of $\mathcal{H}$ containing the
  edge pair.  If twist and orientation coincide for one pair, they coincide for
  the other pair as well. The restriction is unique, if it exists. With respect
	to an affine chart the obtained patch is finite, if and only if $Q$ is finite. 
  \label{lem:adaptedPatch}
\end{lemma}

Finally, we arrive at the definition of our main object of interest
-- affine A-nets extended to piecewise smooth $C^1$-surfaces by gluing
hyperboloid patches into the skew quadrilaterals.

\begin{definition}[Hyperbolic net]
  A generic discrete affine A-net with hyperboloid patches
  associated with the elementary quadrilaterals is a \emph{hyperbolic net}, if
  \begin{enumerate}[i)]
    \item the edges of the quadrilaterals are the bounding asymptotic lines of the
      hyperboloid patches, 
    \item the surface composed of the hyperboloid patches is a piecewise smooth~$C^1$-surface.
  \end{enumerate}
\end{definition}

As mentioned before, there exist affine A-nets that cannot be extended to
hyperbolic nets, even though there exists a 1-parameter family of
pre-hyperbolic nets. The characterizing property for affine A-nets to be
extendable to hyperbolic nets is given in the following definition. 

\begin{definition}[Quad-strip / equi-twisted strips]
  \label{def:stripOrientation}
	A sequence~$\{Q_i\}_{i=1,\dots,n}$ of quadrilaterals in a quadrilateral mesh, such
	that the edges $l_i = Q_{i-1} \cap Q_i$ and $r_i = Q_i \cap Q_{i+1}$ are
	opposite edges in $Q_i$, is called \emph{strip} (cf.  Fig.~\ref{fig:strip}).
	A strip of an A-net is \emph{equi-twisted} if the edge pairs $(l_i,r_i)$
	have the same twist for all skew quadrilaterals along the strip (see Section
	\ref{subsec:orientation_and_twist}).
\end{definition}

\begin{figure}[tbh]
  \begin{center}
    \includegraphics[width=.6\linewidth]{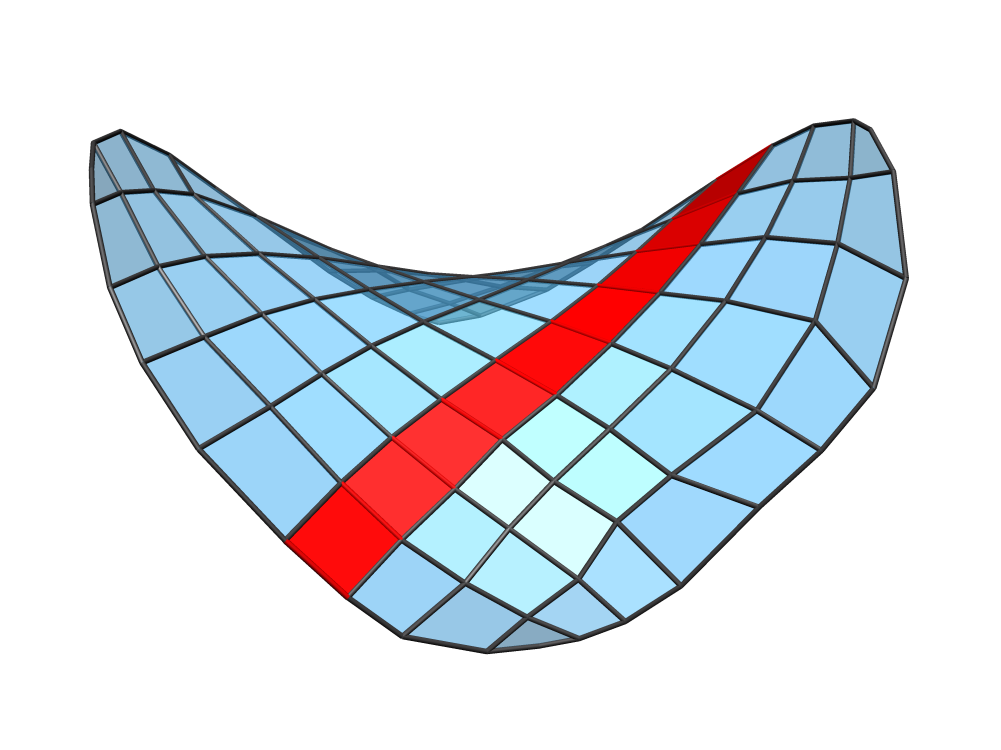}
  \end{center}
  \caption{An equi-twisted strip of an A-net.}
  \label{fig:strip}
\end{figure}

Definition \ref{def:stripOrientation} is motivated by the propagation of the hyperboloids for
pre-hyperbolic nets, since the signatures of corresponding reguli are preserved
by the propagation (see Remark~\ref{rem:propagation_preserves_orientation}).
We obtain the following characterization of those
affine A-nets that can be extended to hyperbolic nets. 

\begin{theorem}
	A simply connected generic affine A-net can be extended to a hyperbolic
	net, if and only if the strips of the A-net are equi-twisted.  In this case
	there exists a 1-parameter family of adapted hyperbolic nets. Each such net is determined
	by the choice of one initial hyperboloid patch that is bounded by one arbitrary
  skew quadrilateral of the A-net.
  \label{thm:extension}
\end{theorem}
\begin{proof}
First observe, that in an A-net with equi-twisted strips all interior vertices
necessarily have even vertex degree, because the two pairs of opposite edges of a
skew quadrilateral have opposite twists. 
The rest follows from Theorem~\ref{thm:nets_of_hyperboloids_from_anets} together with
Lemma~\ref{lem:adaptedPatch} and Remark~\ref{rem:propagation_preserves_orientation}.
\end{proof}

\paragraph{Computer implementation.} The generation of hyperbolic nets has been
implemented in java, using jReality~\cite{jReality}.
Starting with an arbitrary quadrilateral mesh of suitable combinatorics as
input, the first step is to turn the mesh into an A-net. This is achieved by a variational
approach, in which the non-planarity of vertex stars is minimized.
The corresponding functional $\mathcal{S}$ is a sum of volumes of tetrahedra. 
The sum covers all possible tetrahedra~$\Delta$ that are spanned by vertices of the mesh,
under the restriction that for each tetrahedron the four vertices have to be contained
in a single vertex star
\begin{equation*}
\mathcal{S} = 
\sum_{v \in V} 
\sum_{v_1,\ldots,v_4 \in \operatorname{star}(v)} 
\operatorname{Vol}(\Delta(v_1,v_2,v_3,v_4)).
\end{equation*}
To extend the obtained A-net to a hyperbolic net,
we have to choose an initial hyperboloid patch for some quadrilateral. As
described in the previous section we define the initial patch by choosing a point $q \in \plspace$
on the line spanned by the representatives of the two diagonals of the initial skew quadrilateral.
The initial patch is then propagated automatically to all other quadrilaterals and one obtains a
piecewise smooth $C^1$-surface, parametrized as shown in
Fig.~\ref{fig:hypnets}.

\begin{figure}[t]
  \includegraphics[width=.48\linewidth]{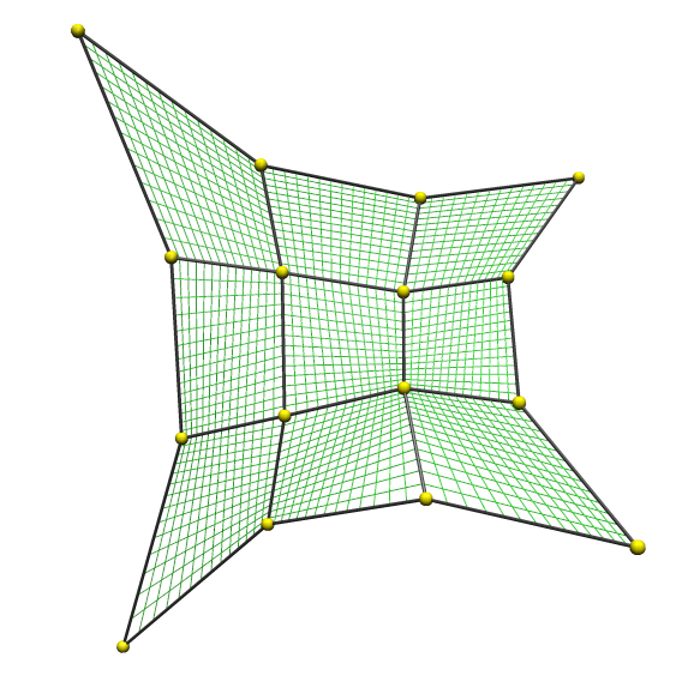}
  \includegraphics[width=.48\linewidth]{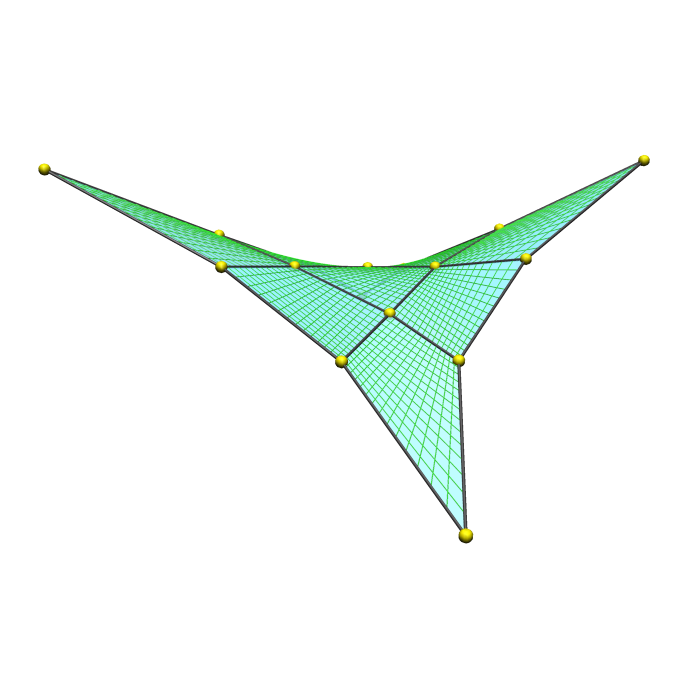}
  \includegraphics[width=.48\linewidth]{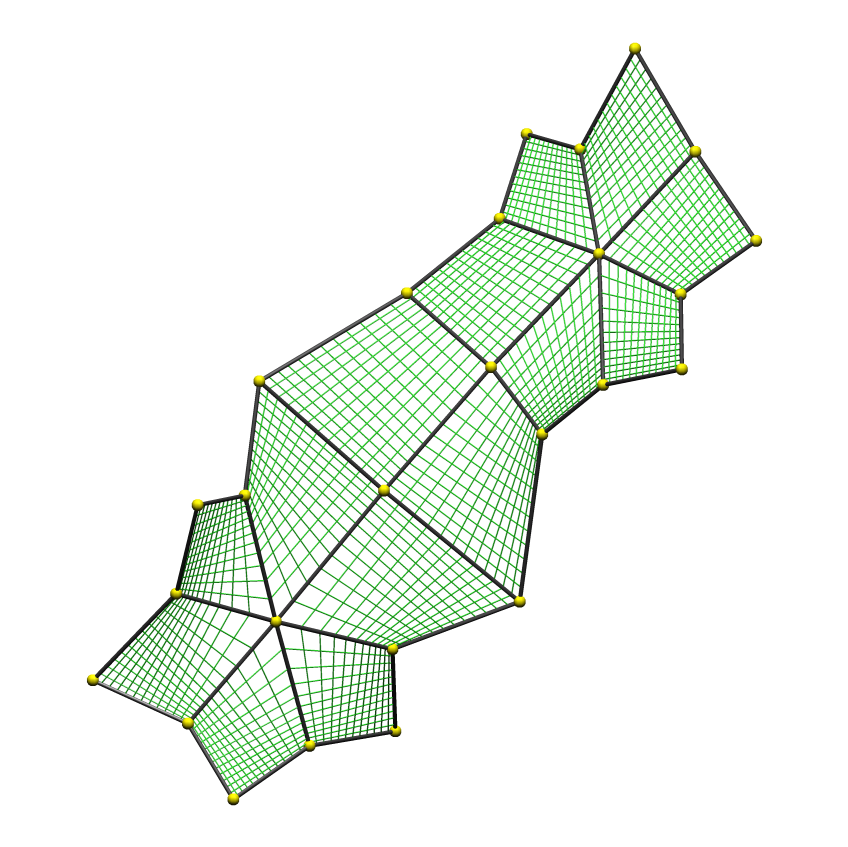}
  \includegraphics[width=.48\linewidth]{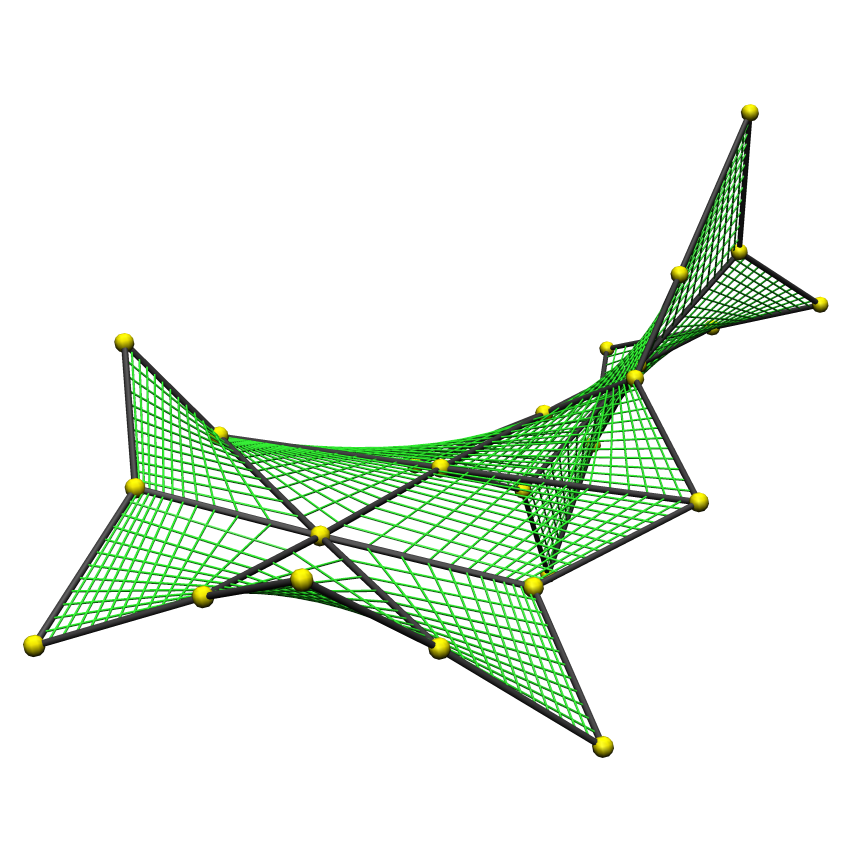}
  \caption{Examples of hyperbolic nets. The top one is of regular $\Z^2$ combinatorics, while the bottom one has two degree six vertices.}
  \label{fig:hypnets}
\end{figure}

\begin{appendix}

\section{}
\label{appendix}

In contrast to the main text, we will now use the exterior algebra model of
Pl\"ucker geometry (see, for example,
\cite{PottmannWallner:2001:ComputationalLineGeometry} or
\cite{BobenkoSuris:2008:DDGBook}) to do explicit calculations with points and
lines.

Using exterior algebra, Pl\"ucker line coordinates arise in a very natural way
if a basis of the space $\R^4$ of homogeneous coordinates for $\RP^3$ has been
chosen.  The space of line coordinates in the exterior algebra description is
$\Lambda^2(\R^4)$, so the projective space containing the Plücker quadric
becomes $\P(\Lambda^2(\R^4))$.  The wedge product on $\R^4$ gives Plücker
coordinates of lines, i.e., the line $\rp h$ spanned by two points $\rp x, \rp
y \in \RP^3$ has Plücker coordinates
\begin{equation*}
h = \rpc x \wedge \rpc y \in \Lambda^2(\R^4).
\end{equation*}
Scaling homogenous coordinates $\rpc x, \rpc y$ by non-zero factors obviously
just results in a corresponding scaling of $h$.

The wedge product on $\Lambda^2(\R^4)$ is exactly the Plücker
product (after canonical identification of $\Lambda^4(\R^4)$ with $\R$).
Two lines $\rp h = \inc{\rp x, \rp y}$ and $\rp l = \inc{\rp u, \rp v}$
in~$\RP^3$ intersect, if and only if
\begin{align*}
&& h \wedge l = (\rpc x \wedge \rpc y) \wedge (\rpc u \wedge \rpc v) &= 0
\quad \text{ in } \Lambda^2(\R^4) \\ 
&\iff&
\langle \pc h , \pc l \rangle &= 0
\quad \text{ in } \R^{3,3}. 
\end{align*}
In particular, for $h \in \Lambda^2(\R^4)$ the condition $h \wedge h = 0$ is
equivalent to $h$ being decomposable. This means that there exist $\rpc x,
\rpc y \in \R^4$, such that $h = \rpc x \wedge \rpc y$.

\paragraph{Orientation of reguli.}
As explained in Section~\ref{subsec:orientation_and_twist}, three skew lines $\rp
h_0, \rp h_1, \rp h_2$ in $\RP^3$ (or $\R^3)$ always span a regulus of
positive or negative orientation.
Let $P=\P(\pc P)$ be the plane in $\P(\Lambda^2(\R^4))$ spanned by the representatives
$[h_0],[h_1],[h_2]$, and denote by
\begin{equation*}
M_P =
 \left( 
  \begin{matrix} 
    0 & h_0 \wedge h_1 & h_0 \wedge h_2 \\ 
    h_0 \wedge h_1 & 0 & h_1 \wedge h_2 \\ 
    h_0 \wedge h_2 & h_1 \wedge h_2 & 0 \\ 
  \end{matrix} 
  \right)
\end{equation*}
the matrix representation of the Plücker product restricted to $\pc P$ with
respect to the basis $h_0,h_1,h_2$.  The \emph{orientation} of $\rp h_0,\rp
h_1,\rp h_2$ is the sign of the determinant of $M_P$.  Since the determinant of
a matrix is the product of its eigenvalues, we have
\begin{equation*}
\text{signature of } P \text{ is } 
\left\{ 
  \begin{array}{c} 
    ({+}{+}{-}) \\
    ({+}{-}{-}) 
  \end{array} 
\right\} 
\iff 
\det(M_P) 
  \left\{ 
    \begin{array}{c} < \\ > 
    \end{array} 
  \right\} 0
\end{equation*}
and one computes
\begin{equation*}
  \det(M_P) = 2 (h_0 \wedge h_1)(h_0 \wedge h_2)(h_1 \wedge h_2).
\end{equation*}
Note, that the sign of $\det(M_P)$ is invariant under rescaling $h_i \mapsto
\lambda_i h_i,\ \lambda_i \in \R \setminus \left\{ 0 \right\}$, of line
coordinates.  After all the orientation of a regulus only depends on the
particular choice of Pl\"ucker coordinates on the space of lines in $\RP^3$,
which in the present case are fixed by using the wedge product.

\paragraph{Twist of skew quadrilaterals.}
In Section \ref{subsec:orientation_and_twist}, the twist of a pair of opposite
edges in a skew quadrilateral $Q = (\rp x, \rp x_1, \rp x_{12}, \rp x_2)$ is
defined.  For example the twist of the edge pair $(\rp x, \rp x_1)$ and $(\rp
x_2, \rp x_{12})$ is the orientation of any regulus spanned by the lines $\rp
h^{(1)} \supset (\rp x, \rp x_1),\ \rp h^{(1)}_2 \supset (\rp x_2, \rp
x_{12})$, and a line~$\rp g^{(1)}$ that is skew to $\rp h^{(1)}$ and $\rp
h^{(1)}_2$, but that intersects the two remaining edges $(\rp x, \rp x_2), (\rp
x_1, \rp x_{12})$ in the interior.  We will now show that the twist of an edge pair is
well defined, i.e., that it is independent of the particular choice of such a line
$\rp g^{(1)}$.  For ease of notation, we write $Q = (\rp a,\rp b,\rp
c,\rp d)$ and denote the three lines in question $\rp h_0,\rp h_1,\rp h_2$.

Choose an affine chart of $\RP^3$ by normalizing homogeneous coordinates to lie
in an affine hyperplane of $\R^4$, such that all edges of $Q$ are finite.
Points on the edges are then described by convex combinations of the
(normalized) homogeneous coordinates $\rpc a, \rpc b, \rpc c, \rpc d$ of
vertices. Let $\rp h_0$ be the line supporting the edge $(\rp a, \rp b)$ and
$\rp h_1$ be the line supporting the edge $(\rp d, \rp c)$. (We use this order
for the vertices on the edges to indicate that in the quadrilateral $(\rp a,\rp
b,\rp c,\rp d)$ the vertex $\rp a$ is connected to $\rp d$ and $\rp b$ is
connected to $\rp c$.) The third line $\rp h_2$ is spanned by two points
$[\alpha \rpc a + (1-\alpha) \rpc d]$ and $[\beta \rpc b + (1-\beta)\rpc c]$,
where $0<\alpha,\beta<1$.  The Pl\"ucker coordinates of the three lines are
given by 
\begin{align*} 
  h_0 &= \rpc a \wedge \rpc b, \\ 
  h_1 &= \rpc d \wedge \rpc c, \\ 
  h_2 &= (\alpha \rpc a + (1-\alpha)\rpc d)\wedge(\beta \rpc b + (1- \beta) \rpc c) \\
  &= 
  \alpha\beta\, \rpc a \wedge \rpc b 
  + \alpha (1-\beta)\,\rpc a \wedge \rpc c 
  + (1-\alpha)\beta \,\rpc d\wedge \rpc b 
  + (1-\alpha)(1-\beta)\,\rpc d\wedge \rpc c.  
\end{align*}

According to the previous consideration about orientation of reguli, the orientation
of $\rp h_0,\rp h_1,\rp h_2$ is the sign of the determinant
\begin{align*} 
  2 (h_0 \wedge h_1)(h_0 \wedge h_2)(h_1 \wedge h_2) &=
  \alpha\beta(1-\alpha)(1-\beta) (\rpc a \wedge \rpc b \wedge \rpc d \wedge \rpc c)^3
\end{align*} 
Since $0<\alpha,\beta<1$, the sign of the expression is exactly the sign of the
product $\rpc a \wedge \rpc b \wedge \rpc d \wedge \rpc c$, and therefore it is
independent of the particular choice of the line $\rp h_2$. We say that the
edges $(\rp a,\rp b)$ and $(\rp d,\rp c)$ have \emph{positive resp.\ negative
twist} depending on this sign.

Note, that the twist is independent of the way of calculation, since the Pl\"ucker
coordinates for lines in $\RP^3$ do not depend on the choice of an affine chart.
In particular, the sign of $\det(M_P)$ does not change under homogenous
rescaling of line coordinates.

The same computation for the lines, say, $\rp g_0$ and $\rp g_1$ spanned by the
other two opposite edges $(\rp a, \rp d)$ and $(\rp b, \rp c)$, and another
line~$\rp g_2$ inbetween, yields that the twist of the second edge pair
$(\rp a,\rp d), (\rp b, \rp c)$ is given by the sign of
\begin{align*}
\rpc a \wedge \rpc d \wedge \rpc b \wedge \rpc c =
- (\rpc a \wedge \rpc b \wedge \rpc d \wedge \rpc c).
\end{align*}
Hence the two pairs of opposite edges have opposite twists. This coincides with
the observation that the two reguli of a hyperboloid have opposite orientations.

\end{appendix}

\paragraph{Acknowledgements.} 
We would like to thank Alexander Bobenko and Wolfgang Schief for fruitful
discussions on the subject and comments on previous versions of the manuscript.
We further want to thank Charles Gunn for providing us with the pictures for
Fig.~\ref{fig:family}.  

\bibliographystyle{amsalpha}
\bibliography{hypnets}

\end{document}